\documentclass[11pt]{amsart}

\usepackage{extsizes}
\usepackage{enumerate}
\usepackage{amsmath}
\usepackage{amsthm}
\usepackage{amssymb}
\usepackage{amsbsy}
\usepackage{amsfonts}
\usepackage{amstext}
\usepackage{amscd}
\usepackage{tikz}
\usepackage{graphicx}
\usetikzlibrary{shapes.geometric}
\usepackage{pgfplots}
\pgfplotsset{width=12cm,compat=1.4}

\numberwithin{equation}{section}
\theoremstyle{plain}
\newtheorem{thm}{Theorem}[section]

\newtheorem{lem}[thm]{Lemma}
\theoremstyle{definition}
\newtheorem{exa}[thm]{Example}
\newtheorem{conj}[thm]{Conjecture}

\newtheorem{defi}[thm]{Definition}

\newcommand{\real}{\mathbb{R}}
\DeclareMathOperator*{\comp}{\mathbb{C}}
\DeclareMathOperator*{\nat}{\mathbb{N}}
\newcommand{\trr}{\triangleright}
\newcommand{\iu}{\mathcal{IU}}
\newcommand{\id}{\mathcal{ID}}
\newcommand{\nn}{\mathbb{N}}
\newcommand{\E}{\mathbb{E}}
\newcommand{\R}{\mathbb{R}}
\newcommand{\Pb}{\mathbb{P}}
\newcommand{\Mod}[1]{\ (\mathrm{mod}\ #1)}

\calclayout

\begin{document}
	\title [ Trees are hypoenergetic] {Barabasi-albert trees are hypoenergetic}
	
	
	
	\author[Arizmendi]{Octavio Arizmendi}
	\author[Dominguez]{Emilio Dominguez}

	\keywords{Graph Energy, Non Commutative Probability, Graphs Spectra, Inequalities}
	\subjclass[2010]{05C50,05C80}
	\address{Centro de Investigaci{\'o}n en Matem{\'a}ticas. Guanajuato, M\'exico}
	\email{octavius@cimat.mx}
	
	\address{Universidad Aut\'onoma de Sinaloa. Sinaloa, M\'exico}
	\email{jedguez@gmail.com}
	\keywords{Barabasi Albert tree, Graph Energy}
	
	\maketitle
	\begin{abstract}
		We prove that graphs following the model of Barabasi-Albert tree with $n$ vertices are hypoenergetic in the large $n$ limit. 
	\end{abstract}

	\section{Introduction and Statements of Results}
	
	The {\it graph energy} is a graph invariant that was defined by I. Gutman \cite{Gut2} from his studies of mathematical chemistry. 

The energy of a graph $G$, denoted by $\mathcal{E}(G)$, is defined as the sum of the absolute values of the eigenvalues of the adjacency matrix $A = A(G)$, i.e.,
	\begin{equation*}
	\mathcal{E}(G) = \sum_{i=1}^{n}|\lambda_{i}|.
	\end{equation*}
In other words, the energy of a graph is the trace or nuclear norm of its adjacency matrix.

Many results on inequalities for $\mathcal{E}(G)$ have been established and there are many examples of deterministic graphs whose energy is known. An excellent introduction to the theory of graph energy can be found in the monograph \cite{LiShiGu}.
However, very few research papers have studied the properties of energy on random graphs. 

To the best of our knowledge the only papers considering the energy of a random graph are due Nikiforov \cite{Nik,Nik2}. In these papers, Nikiforov describes precisely the asymptotic behavior, as the size of the graph goes to infinity, of the energy of two families of random graphs: Erd\"os-R\'enyi graph with fixed $p$, \cite{Nik}, and uniform $d$-regular graphs, \cite{Nik2}. Both of these results rely on the fact that the explicit limiting distributions of the adjacency matrix of these graphs are well known.  

This paper contributes to the theory of graph energy in two ways. Firstly, by the use of Ky Fan's inequality we propose a new inequality for the energy of a tree in Theorem \ref{T2}, which in some cases dramatically improves known inequalities such as McClelland's, Koolean and Moulton's and their variants. Secondly,  we use the previous inequality and modify its proof to show that Barabasi-Albet trees are asymptotically hypoenergetic, (i.e. the energy is smaller than the number of vertices), as the size of the matrix goes to infinity. We state this second result as our main theorem. 

\begin{thm}
With probability tending to $1$, as the size tends to infinity, the energy of a typical tree chosen with the preferential model of Barabasi-Albert is smaller than the size of the tree. \label{T1}
\end{thm}

Let us mention that the results of \cite{Nik}, show that almost all graphs on $n$ vertices have energy of order $(\frac{4}{
3\pi}+o(n))
n^3/2$, and in particular are hyperenergetic (i.e their energy is larger than the energy of the complete graph). Theorem \ref{T1} constrasts with this and shows that there are plenty of trees of size $n$, with energy smaller than $n$. This should not be as surprising since trees or tree-like graphs are very rare when considering uniform random graphs. We believe that in order to compare Erd\"os-R\'enyi graphs with random trees one needs to set $p=\frac{2}{n}$, so that at least the expected number of edges coincides. We try to build some intuition in the last section of the paper.

Apart from this introduction, the paper contains 4 more sections. Section 2 gives the necessary preliminaries. Section 3 proves our new inequality for the energy of a tree. The main theorem of the paper is proved in Section 4. We present some simulations, leading to open problems and conjectures in the final Section 5.

\section{Preliminaries}

\subsection{Notation on graphs}

We consider finite simple undirected graphs. For definitions used here we refer to Diestel \cite{Diestel}.  A graph is called a tree if it is connected and has no cycles.

For a graph $G=(V,E)$ and a vertex, $v\in V$, the closed neigborhood of $v$ is the set $N(v)=\{v\}\cup\{w\in V| w\sim v\}$, and the degree of $v$, denoted by $d(v)=deg(v) $ is the number of neighbors of $v$, i.e.  $d(v)=|N(v)|-1$.

A tree is called a star graph if it consists of one vertex which is joined to $n-1$ vertices. We denote the star graph with $n$ vertices by $S_n$, which is also called the $n$-star.

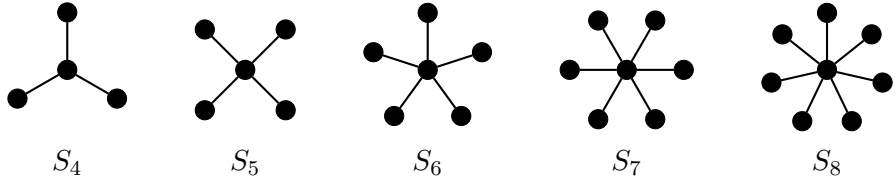
\begin{figure}[h] \label{stars}
\begin{tikzpicture}
[mystyle/.style={scale=0.7, draw,shape=circle,fill=black}]
\def\ngon{3}
\node[regular polygon,regular polygon sides=\ngon,minimum size=1.5cm] (p) {};
\foreach\x in {1,...,\ngon}{\node[mystyle] (p\x) at (p.corner \x){};}
\node[mystyle] (p0) at (0,0) {};
\foreach\x in {1,...,\ngon}
{
 \draw[thick] (p0) -- (p\x);
}
  \node [label=below:$S_{4}$] (*) at (0,-0.8) {};
 \end{tikzpicture}
  \qquad
\begin{tikzpicture}
[mystyle/.style={scale=0.7, draw,shape=circle,fill=black}]
\def\ngon{4}
\node[regular polygon,regular polygon sides=\ngon,minimum size=1.5cm] (p) {};
\foreach\x in {1,...,\ngon}{\node[mystyle] (p\x) at (p.corner \x){};}
\node[mystyle] (p0) at (0,0) {};
\foreach\x in {1,...,\ngon}
{
 \draw[thick] (p0) -- (p\x);
}

  \node [label=below:$S_{5}$] (*) at (0,-0.8) {};
 \end{tikzpicture}
  \qquad
\begin{tikzpicture}
[mystyle/.style={scale=0.7, draw,shape=circle,fill=black}]
\def\ngon{5}
\node[regular polygon,regular polygon sides=\ngon,minimum size=1.5cm] (p) {};
\foreach\x in {1,...,\ngon}{\node[mystyle] (p\x) at (p.corner \x){};}
\node[mystyle] (p0) at (0,0) {};
\foreach\x in {1,...,\ngon}
{
 \draw[thick] (p0) -- (p\x);
}
  \node [label=below:$S_{6}$] (*) at (0,-0.8) {};
 \end{tikzpicture}
  \qquad
\begin{tikzpicture}
[mystyle/.style={scale=0.7, draw,shape=circle,fill=black}]
\def\ngon{6}
\node[regular polygon,regular polygon sides=\ngon,minimum size=1.5cm] (p) {};
\foreach\x in {1,...,\ngon}{\node[mystyle] (p\x) at (p.corner \x){};}
\node[mystyle] (p0) at (0,0) {};
\foreach\x in {1,...,\ngon}
{
 \draw[thick] (p0) -- (p\x);
}
  \node [label=below:$S_{7}$] (*) at (0,-0.8) {};
 \end{tikzpicture}
  \qquad
  \begin{tikzpicture}
[mystyle/.style={scale=0.7, draw,shape=circle,fill=black}]
\def\ngon{7}
\node[regular polygon,regular polygon sides=\ngon,minimum size=1.5cm] (p) {};
\foreach\x in {1,...,\ngon}{\node[mystyle] (p\x) at (p.corner \x){};}
\node[mystyle] (p0) at (0,0) {};
\foreach\x in {1,...,\ngon}
{
 \draw[thick] (p0) -- (p\x);
}
  \node [label=below:$S_{8}$] (*) at (0,-0.8) {};
 \end{tikzpicture}
  \caption{Stars $S_4$, $S_5$, $S_6$, $S_7$ and $S_8$}
\end{figure}

A tree is called a \emph{path} if it consists of $n$ vertices, $v_1,\dots,v_n$, such that $v_i\sim v _j$. We denote the path of size $n$ by $P_n$.

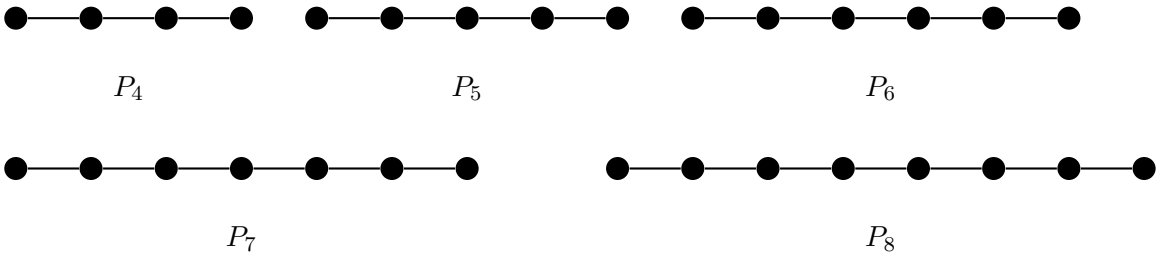
\begin{figure}
[h]\centering 
  \begin{tikzpicture}

\node[style={circle,fill=black,inner sep=2pt]}] (27) at (1,7) {$.$};
\node[style={circle,fill=black,inner sep=2pt]}] (37) at (2,7) {$.$};
\node[style={circle,fill=black,inner sep=2pt]}] (47) at (3,7) {$.$};
\node[style={circle,fill=black,inner sep=2pt]}] (57) at (4,7) {$.$};

\path[draw,thick](27) edge node {} (37);
\path[draw,thick](37) edge node {} (47);
\path[draw,thick](47) edge node {} (57);

  \node [label=below:$P_{4}$] (*) at (2.5,6.5) {};

\node[style={circle,fill=black,inner sep=2pt]}] (26) at (5,7) {$.$};
\node[style={circle,fill=black,inner sep=2pt]}] (36) at (6,7) {$.$};
\node[style={circle,fill=black,inner sep=2pt]}] (46) at (7,7) {$.$};
\node[style={circle,fill=black,inner sep=2pt]}] (56) at (8,7) {$.$};
\node[style={circle,fill=black,inner sep=2pt]}] (66) at (9,7) {$.$};

\path[draw,thick](26) edge node {} (36);
\path[draw,thick](36) edge node {} (46);
\path[draw,thick](46) edge node {} (56);
\path[draw,thick](56) edge node {} (66);

  \node [label=below:$P_{5}$] (*) at (7,6.5) {};

\node[style={circle,fill=black,inner sep=2pt]}] (25) at (10,7) {$.$};
\node[style={circle,fill=black,inner sep=2pt]}] (35) at (11,7) {$.$};
\node[style={circle,fill=black,inner sep=2pt]}] (45) at (12,7) {$.$};
\node[style={circle,fill=black,inner sep=2pt]}] (55) at (13,7) {$.$};
\node[style={circle,fill=black,inner sep=2pt]}] (65) at (14,7) {$.$};
\node[style={circle,fill=black,inner sep=2pt]}] (75) at (15,7) {$.$};

\path[draw,thick](25) edge node {} (35);
\path[draw,thick](35) edge node {} (45);
\path[draw,thick](45) edge node {} (55);
\path[draw,thick](55) edge node {} (65);
\path[draw,thick](65) edge node {} (75);

  \node [label=below:$P_{6}$] (*) at (12.5,6.5) {};

\node[style={circle,fill=black,inner sep=2pt]}] (34) at (1,5) {$.$};
\node[style={circle,fill=black,inner sep=2pt]}] (44) at (2,5) {$.$};
\node[style={circle,fill=black,inner sep=2pt]}] (54) at (3,5) {$.$};
\node[style={circle,fill=black,inner sep=2pt]}] (64) at (4,5) {$.$};
\node[style={circle,fill=black,inner sep=2pt]}] (74) at (5,5) {$.$};
\node[style={circle,fill=black,inner sep=2pt]}] (84) at (6,5) {$.$};
\node[style={circle,fill=black,inner sep=2pt]}] (94) at (7,5) {$.$};
\path[draw,thick](34) edge node {} (44);
\path[draw,thick](44) edge node {} (54);
\path[draw,thick](54) edge node {} (64);
\path[draw,thick](64) edge node {} (74);
\path[draw,thick](74) edge node {} (84);
\path[draw,thick](84) edge node {} (94);

  \node [label=below:$P_{7}$] (*) at (4,4.5) {};

\node[style={circle,fill=black,inner sep=2pt]}] (32) at (9,5) {$.$};
\node[style={circle,fill=black,inner sep=2pt]}] (42) at (10,5) {$.$};
\node[style={circle,fill=black,inner sep=2pt]}] (52) at (11,5) {$.$};
\node[style={circle,fill=black,inner sep=2pt]}] (62) at (12,5) {$.$};
\node[style={circle,fill=black,inner sep=2pt]}] (72) at (13,5) {$.$};
\node[style={circle,fill=black,inner sep=2pt]}] (82) at (14,5) {$.$};
\node[style={circle,fill=black,inner sep=2pt]}] (92) at (15,5) {$.$};
\node[style={circle,fill=black,inner sep=2pt]}] (102) at (16,5) {$.$};

\path[draw,thick](32) edge node {} (42);
\path[draw,thick](42) edge node {} (52);
\path[draw,thick](52) edge node {} (62);
\path[draw,thick](62) edge node {} (72);
\path[draw,thick](72) edge node {} (82);
\path[draw,thick](82) edge node {} (92);
\path[draw,thick](92) edge node {} (102);
  \node [label=below:$P_{8}$] (*) at (12.5,4.5) {};

 \end{tikzpicture}
  \caption{Paths $P_4$, $P_5$, $P_6$, $P_7$ and $P_8$}
\end{figure}

\subsection{Inequalities for the Energy of a Graph}

In this section we remind the reader of some inequalities for the energy of a graph. 

The first inequality is the well known McClelland's inequality \cite{McC} which says that
$$\mathcal{E}G)\leq \sqrt{2mn}. $$

Two other inequalities due to Koolen and Moulton are the following. First, in \cite{KM} they prove that for any graph with $n$ vertices and $m$ edges we have
\begin{equation} \label{KM1}\mathcal{E}(G)\leq\frac{2m}{n}+\sqrt{(n-1)(2m-(\frac{2m}{n})^2)}.\end{equation}
If, moreover one knows that $G$ is bipartite graph, from \cite{KM2}, one has 
\begin{equation} \label{KM2} \mathcal{E}(G)\leq2\frac{2m}{n}+\sqrt{(n-2)(2m-2(\frac{2m}{n})^2)}.\end{equation}

The following theorem of Ky-Fan will be very useful, which in particular shows that the graph energy satisfies a triangle inequality.
 
\begin{thm}\label{KF} Let $A, B$, and $C$ be square selfadjoint matrices of order $n$, such that $A+B=C.$
Then
$$\sum^n_i |\lambda_i(A)|+\sum^n_i |\lambda_i(B)|\geq\sum^n_i |\lambda_i(C)|$$
\end{thm}

We state a version of Ky-Fan's inequality for subgraphs.
\begin{thm}[Ky-Fan's inequality for subgraphs] \label{KyFanthm}
		Let $G$ be a graph and $H_1,H_2,...,H_n$ be subgraphs of $G$ whose adjacency matrices satisfy the condition $A(G)=A(H_1)+A(H_2)+\cdots+A(H_n)$. Then
	\begin{equation}\label{kyfanine} \mathcal{E}(G)\leq  \mathcal{E}(H_1)+\mathcal{E}(H_2)+\cdots+\mathcal{E}(H_n)\end{equation}
\end{thm}

\begin{defi}
Let $G=(V,E)$ be a simple graph. A \emph{graph partition} is a collection $\mathcal{G}=\{G_i=(V_i,E_i)\}_{i\in I}$ of subgraphs of $G$ with the following properties.

\begin{enumerate}
\item For any edge $e\in E$, there is $G_i\in \mathcal G$ such that $e\in E_i$ 
\item Any $G_i,G_j\in \mathcal G $ with $G_i\neq G_j$ do not share edges. i.e $E_i\cap E_j$ is empty.
\end{enumerate}
\end{defi}

Note that graph partitions satisfy the hypotheses of the  Ky-Fan's inequality. The following lemma is a direct consequence of this.
\begin{lem}
Let $\mathcal{G}$ be a graph partition of $G$ then
$$\mathcal{E} (G)\leq \sum_{H\in \mathcal{G}} \mathcal{E} (H) $$
\end{lem}

Finally,  we would like to mention the following inequality found by Arizmendi and Juarez \cite{ArJu} in their study of the \emph{energy of a vertex}.  We give a new simple proof by the use of  Ky-Fan's theorem, Theorem \ref{KF}.
\begin{lem}[Arizmendi \& Juarez \cite{ArJu}]
For a graph $G$ with vertices of degrees $d_1,\cdots,d_n$ 	
\begin{equation}\label{bound1} \mathcal{E}(G)\leq  \sum^n_{i=1} \sqrt{d_i}\end{equation}
\end{lem}

\begin{proof}
For each vertex $i$, consider the selfajoint matrix $A^{(i)}$ such that $A^{(i)}_{ji}=A^{(i)}_{ij}=1/2$ if $i\sim j$,  $A^{(i)}_{ji}=A^{(i)}_{ij}=0$ if $j\nsim i$ and $A^{(i)}_{lk}=0$ if $l,k\neq i$. 

Since $2A^{(i)}$ is the adjacency matrix of a $d_i$-star graph, $ \sum^n_{k=1}|\lambda_k(A^{(i)})|=\sqrt{d_i}$. Finally, since  $A^{(i)}_{ij}+A^{(j)}_{ij}=1$, one easily sees that $A(G)=\sum^n_{i=1} A^{(i)}$. We conclude by Ky-Fan's theorem.
\end{proof}

\subsection{Barabasi-Albert Model}

Various natural, social and technological systems are thought to have a scale free degree distribution: the proportion $P(d)$ of nodes with degree $d$ is proportional to $d^{-\gamma}$, at least asymptotically, for some $\gamma$ that does not depend on the size of the network. To address this, in \cite{ABo} Barabasi and Albert proposed a random graph model that grows with time and incorporates preferential attachment. That is, a new vertex attaches with higher probability to already existing vertices of higher degrees. The original construction of the model is as follows:\\

Starting with $m_0$ vertices, at each step we add a new vertex and $m(\leq m_0)$ edges, joining the new vertex to $m$ vertices already in the graph. These edges are randomly chosen in such a way that the probability for the new vertex to be connected to am already existing vertex $i$ is proportional to the degree of $d_i$, $$\Pi(d_i)=\frac{d_i}{\sum_j d_j}.$$  

We are concerned with the case where both $m_0$ and $m$ are equal to $1$. Here, the graph obtained by the process in every time step is a tree. A parameter $\alpha$ is added to the model in order to consider the case of non-linear dependency for the preferential attachment, as seen in Krapivsky et al. \cite{KRL}. In this case, the probability for a new vertex to be connected to a vertex $i$ of degree $d_i$ is now proportional to $d_i^\alpha$. That is to say, $\Pi(d_i)=d_i^\alpha/\sum_{j} d_j^\alpha$. 
	
\section{A basic bound for trees}\label{Bounds}
	
	One basic problem is to find the extremal values or good bounds for the energy within some special class of graphs and to characterize graphs from this class which reach this extremal values of the energy. In this section we present an upper bound for the energy of a tree, which can be derived rather simply from the theory, specifically from Ky Fan's inequality, but a priori not so natural. This bound turns out to be quite useful specially for trees with many vertices of degree $1$, which is where the Arizmendi-Juarez bound, \eqref{bound1}, seems to perform badly. This will be used in the next section when considering Barabasi-Albert trees.
	
	\begin{thm}\label{T2} Let $T$ be a tree with degrees $\Delta=d_1\geq....\geq d_n$, $n\geq3$. Then 
		\begin{equation} \label{bound2} \mathcal{E}(T)\leq  \sum^n_{i=2} 2\sqrt{d_i-1}+2\sqrt{\Delta}\leq \sum^n_{i=1} 2\sqrt{d_i-1}+1\end{equation}
	\end{thm}
	
	\begin{proof}
	 Let $v_0$ be a vertex with largest degree. We root the tree $T$ at $v_0$.  
For each vertex $v\neq v_o$, there is a unique path $v_0\tilde v_1\tilde \cdots \tilde v_k=v$ from $v_0$ to $v$. Let $a(v)=v_{k-1}$ be the last vertex in this path before arriving to $v$, which is sometimes called the parent or ancestor of $v$. Now, let $\hat N(v)=v\cup N(v)\setminus \{a(v)\}$. We denote by $H(v)$, the induced subgraph of $G$, with vertex set $\hat N(v)$. Notice that $H(v)$ is a star $S_{d(v)}$ so its energy is given by $2\sqrt{d(v)-1}$. We also denote by $H(v_0)$ the subgraph induced by $v_0\cup N(v_0)$, which is a star $S_\Delta$ with energy $2\sqrt{\Delta}$ 

	 The family $\{H(v)\}_{v\in V}$, is a graph partition of $T$. By the use of Ky-Fan's inequality  \eqref{kyfanine} we may deduce the first inequality.
		
	The second inequality holds since for $\Delta\geq2$, $\sqrt{\Delta}-\sqrt{\Delta-1}\leq\frac{1}{2\sqrt{\Delta-1}}\leq1/2.$		
	\end{proof}

\begin{exa} A tree is called a double star $S_{p,q}$ if it is obtained by joining the centers of two
stars $S_p$ and $S_q$ by an edge. The double star $S_{p,q}$ has $p+q$ vertices. Its characteristic polynomial is given by  $\chi_{p,q}(x)=x^n-(p+q-1)x^{n-2}+(p-1)(q-1)x^{n-4}$ and thus the $4$ non -zero eigenvalues of $S_{p,q}$ are 
$\pm\frac{1}{\sqrt{2}}\sqrt{p+q-1\pm\sqrt{(p+q+1)^2-4(pq+1)}}$.  Hence we can easily calculate the energy: $$\mathcal{E}(S_{p,q})=\sqrt{2}\left(\sqrt{p+q-1+\sqrt{(p+q+1)^2-4(pq+1)}}+\sqrt{p+q-1-\sqrt{(p+q+1)^2-4(pq+1)}}\right)$$

In order to see the improvement of our bound, the following table compares the true values of the energy of the double star $S_{p,q}$, for $p=5$ and different values of $q$,  with the bounds from McClelland Inequality, the bound from Arizmendi-Juarez (A-J, Eq.\eqref{bound1}), Koolen and Moulton Inequalities (K-M 1, Eq.\eqref{KM1} and K-M 2, Eq. \eqref{KM2}) and the above bound Eq. \eqref{bound2}.

\begin{center}
	\centering
	\begin{tabular}{|c|c|c|c|c|c|c|} 
		\hline
		\multicolumn{1}{|c|}{$ p=5$ } & \multicolumn{1}{c|}{Energy} & \multicolumn{1}{c|}{\textbf{Thm. \ref{T2}} } & \multicolumn{1}{c|}{K-M 1} & \multicolumn{1}{c|}{K-M 2} & \multicolumn{1}{c|}{A-J} & \multicolumn{1}{c|}{McClelland}  \\ 
		\hline
		$q=1$                        & 4.472                       & 4.472                                                                   & 7.676             &  7.550               & 6.                                    & 7.745                            \\
		$q=2$                        & 6.324                       & 6.472                                                                   & 9.088             &  8.961               & 8.                                    & 9.165                            \\
		$q=3$                        & 7.115                       & 7.3                                                                     & 10.500            &  10.374               & 9.414                                 & 10.583                           \\
		$ q=4$                       & 7.727                       & 7.936                                                                   & 11.913            &  11.787               & 10.732                                & 12.                              \\
		$ q=5$                       & 8.246                       & 8.472                                                                   & 13.326            &  13.200               & 12.                                   & 13.416                           \\
		$q=6$                        & 8.705                       & 8.944                                                                   & 14.739            &  14.613               & 13.236                                & 14.832                           \\
		$ q=7$                       & 9.120                       & 9.371                                                                   & 16.152            &  16.027               & 14.449                                & 16.248                           \\
		$ q=8$                       & 9.504                       & 9.763                                                                   & 17.566            &  17.441               & 15.645                                & 17.663                           \\
		$ q=9$                       & 9.861                       & 10.129                                                                  & 18.979            &  18.854               & 16.828                                & 19.078                           \\
		$ q=10$                      & 10.198                      & 10.47                                                                   & 20.393            &  20.268               & 18.                                   & 20.493                           \\
		\hline
	\end{tabular}
\end{center}

\end{exa}

\begin{exa}
Similar as for the double star we compare the different inequalites for the path of size $n$, $P_n$, with the energy given by,
$$\mathcal{E}(P_{2n})= \frac{2}{\sin(\pi /(4n+2)}-2 , \qquad  \mathcal{E}(P_{2n+1})= \frac{2\cos(\pi /(4n+4)}{\sin(\pi /(4n+4)}-2. $$
\begin{center}
	\centering
	\begin{tabular}{|c|c|c|c|c|c|c|} 
		\hline
		\multicolumn{1}{|c|}{  } & \multicolumn{1}{c|}{Energy} & \multicolumn{1}{c|}{\textbf{Thm. \ref{T2}} } & \multicolumn{1}{c|}{K-M 1} & \multicolumn{1}{c|}{K-M 2} & \multicolumn{1}{c|}{A-J} & \multicolumn{1}{c|}{McClelland}  \\ 
		\hline
		$n=2$                        & 2.              & 2.                                                                  &2.             & 2.               & 2.                                    & 2.                           \\
		$n=3$                        & 2.828                       & 2.828                                                                     & 3.441            & 3.333             & 3.414                               & 3.464                           \\
		$ n=4$                       & 4.472                       & 4.828                                                                   & 4.854            &  5.732               & 4.828                                 & 4.898                              \\
		$ n=5$                       & 5.464                       & 6.828                                                                   & 6.264            &  6.139            & 6.242                                   & 6.324                           \\
		$n=6$                        & 6.988                      & 8.828                                                                    & 7.676             &  7.550               &  7.656                                & 7.745                            \\
		$ n=7$                       &   8.055                     & 10.828                                                                   & 9.088            &  8.961               & 9.071                               & 9.165                          \\
		$ n=8$                       & 9.517                      & 12.828                                                                   & 10.500            &  10.374               & 10.485                                & 10.583                           \\
		$ n=9$                       & 10.627                       & 14.828                                                                  & 11.913            &  11.787               & 11.899                                & 12.                          \\
		$ n=10$                      & 12.053                   & 16.828                                                                   & 13.326           &  13.200               & 13.313                                   & 13.416                           \\
		\hline
	\end{tabular}
\end{center}

\end{exa}

As one can see from the above example, Theorem 3.1 performs badly, compared with the rest of the inequalities in cases where there are many vertices of degree $2$ joined by an edge. In this case, compared with the bound \eqref{bound1}, for each pair we get a contribution of $2$ instead of $\sqrt{2}$. This observation will be useful in the next section in order to improve the bound \eqref{bound2}, when proving Theorem \ref{T1}.

\section{An upper bound for BA trees}

In this section we prove the main theorem, Theorem \ref{T1}. The proof consists of modifying the partition used in the proof of Theorem \ref{T2} by analizing more carefully vertices of degree 2. Before doing that, we estimate what Theorem \ref{T2} says about Barabasi-Albert trees.
\subsection{A first bound}	
Let $T$ be a tree of order $n$ that follows the model of Barabasi-Albert with parameter $\alpha = 1$.

For each $d\geq1$, we denote by $n_d$ the amount of vertices of degree $d$ in $T$. Before going into details, we notice that the total number of vertices may be calculated by
\begin{equation}\label{1}\sum_{d=1}^n n_d=n,\end{equation}
while the sum of degrees may be calculated by 
\begin{equation}\label{2}\sum_{d=1}^n n_d \cdot d=2n-2.\end{equation}
	
Now, in the case of a Barabasi-Albert tree, if we denote by
$$ \alpha_{d} = \frac {4} {(d) (d + 1) (d + 2)}, $$ 
as showed by Bollob\'as et al.  \cite{BRST}, for every fixed $d$, with probability tending to $1$ as $n\to\infty$, we have
$$ \frac {n_d} {n}=\alpha_ {d}+o(1).$$
	
Let $\Delta=d_1\geq....\geq d_n$ be the degrees of the vertices in $T$. We look at the sum 
$$S=\sum_{i=1}^n 2\sqrt{d_i-1}=2\sum_{d=1}^{n}n_d\cdot\sqrt{d-1}.$$ 
Using that for any fixed $m \in \mathbb {N} $
$$\sum^m_{d=1}{\frac{4}{(d)(d+1)(d+2)}}=2\sum^m_{d=1}{ \left( \frac{1}{d}+\frac{1}{d+2}-\frac{2}{d+1} \right)}=1-\frac{2}{(m+1)(m+2)}$$
we see that for any $m$ fixed and large $n$, from \eqref{1} 
$$\sum^n_{d=m+1}\frac{n_d}{n}=\frac{2}{(m+1)(m+2)}+o(1)$$
or 
$$\sum^n_{d=m+1}n_d=n\frac{2}{(m+1)(m+2)}+o(n),$$
from where, together with \eqref{2}, by Cauchy-Schwarz we see that
$$\sum^{n}_{d=m+1}{n_d\sqrt{d}} \leq  \sqrt{(2n-2)(n)\frac{2}{(m+1)(m+2)}} + o(n)\leq n\frac{2}{m} + o(n).$$
	
By using that $\frac{n_d}{n}=\alpha_{d}+o(1)$ for $1\leq d\leq m$, we get the relations: 
	\begin{align*}
	\frac{S}{n}&=2\sum^n_{d=1}{\frac{n_d}{n} \sqrt{d-1}}\\
	&=2 \sum^{m}_{d=1}{\frac{4\sqrt{d-1}}{(d)(d+1)(d+2)}}+o(1)+ \sum^{n}_{d=m+1}{\frac{n_d}{n} \sqrt{d-1}} \\		
	&\leq 2\sum^{m}_{d=1}{\frac{4\sqrt{d-1}}{(d)(d+1)(d+2)}}+o(1)+\frac{1}{n}\sum^{n}_{d=m+1}{n_d\sqrt{d} }\\
	&\leq 2\sum^{m}_{d=1}{\frac{4\sqrt{d-1}}{(d)(d+1)(d+2)}}+\frac{2}{m} + o(1)\\
	&\leq 2\sum^{\infty}_{d=1}{\frac{4\sqrt{d-1}}{(d)(d+1)(d+2)}}+\frac{2}{m} + o(1).
	\end{align*}
	which holds for any fixed $m$ and large $n$.
	
Since $m$  is fixed but arbitrary, we obtain the asymptotic bound
$$\frac{S}{n}\leq2\sum^{\infty}_{d=1}{\frac{4\sqrt{d-1}}{(d)(d+1)(d+2)}}\approx 1.00576755.$$

\subsection{Proof of Theorem \ref{T1}}	
To achieve our goal we need to look a little bit further on the degrees of the tree. So let us denote by $n_{kl}$ the proportion of edges that connect a pair of nodes of degrees $k$ and $l$ with $k\leq l$ in our graph $G$. As showed in \cite{BRST}, with probability tending to $1$, as $n\to\infty$ we have
	\begin{align*}
		n_{kl}=&\frac{4(l-1)}{k(k+1)(k+l)(k+l+1)(k+l+2)}\\
		+&\frac{12(l-1)}{k(k+l-1)(k+l)(k+l+1)(k+l+2)}+o(1).
	\end{align*}
In particular, $n_{2,2}=\frac{1}{45}+o(1)$.
	
	Label the vertices of $G$ as $v_1,\cdots,v_n$ so that $v_1$ is a vertex of highest degree $\Delta$ in $G$ and partition the graph into a star $S_{\Delta+1}$ centered at a $v_1$ and $n-1$ stars $S_{d_i}$ centered in each of the other vertices as in Theorem \ref{T2}. For each vertex $v$, denote the corresponding star with center in $v$ by $H(v)$. 
	
	Let $u_1,u_2$ be adjacent vertices of degree $2$ such that $H(u_1)$ is the graph induced by this pair of vertices. Let $u_0,u_3$ be the only other neighbors of $u_1$ and $u_2$, respectively. We change the partition of the graph by taking out $H(u_1),H(u_2)$ from it and replacing them with their union. As both graphs are stars with one edge and they share a vertex, its union is a star $S_3$. The energy of $S_3$ is $2\sqrt{2}$, opposed to the sum of energies of $H(u_1),H(u_2)$ which is $4$, so the sum of energies of all parts after this change in the partition is reduced by $(4-2\sqrt{2})$. We can repeat this process for each other pair of neighbors of degree $2$ as long as they are different from any previously chosen pairs $\lbrace u_1,u_2\rbrace$ and from previous pairs of the form $\lbrace u_0,u_1\rbrace$ or $\lbrace u_2,u_3\rbrace$.
	
	Each time we iterate this process, at most $3$ pairs of neighbors of degree $2$ become unusable. Therefore, we are able to do this replacement at least $\frac{1}{3}n\cdot n_{2,2}$ times. After iterating as many times as possible, we get a partition of $G$ into graphs $K_1,\cdots,K_m$ such that:
	\begin{align*}
	\frac{\mathcal{E}(G)}{n} \leq&\frac{1}{n}\sum_{j=1}^{m}{\mathcal{E}(K_j)}\\
	\leq&\frac{1}{n}\sum_{i=1}^{n}{\left(\mathcal{E}(H(v_i))\right)}-\frac{n_{2,2}}{3}(4-2\sqrt{2})\\
	=&\frac{1}{n}\left(\sum^n_{i=2} 2\sqrt{d_i-1}+2\sqrt{\Delta}\right)-\frac{4-2\sqrt{2}}{3\cdot45}+o(1)\\
	\leq&\frac{S}{n}+\frac{1}{n\sqrt{\Delta-1}}-\frac{4-2\sqrt{2}}{3\cdot45}+o(1)\\
	\leq& 2\sum^{\infty}_{d=1}{\frac{4\sqrt{d-1}}{(d)(d+1)(d+2)}}-\frac{4-2\sqrt{2}}{135}+o(1)\\
	\approx& 0.997089+o(1)
	\end{align*}

This proves that with probability tending to $1$ as $n\to\infty$, a graph $G$ of $n$ vertices that follows the model of Barabasi-Albert is hypoenergetic.

Finally, let us notice that the  above considerations actually work for any tree,  resulting in the following
\begin{thm}\label{T4} Let $T$ be a tree with degrees $\Delta=d_1\geq....\geq d_n$, $n\geq3$. Then 
		\begin{equation} \label{bound3} \mathcal{E}(T)\leq  \sum^n_{i=2} 2\sqrt{d_i-1}+2\sqrt{\Delta}-\frac{e_{2,2}}{3}(4-2\sqrt{2})
\end{equation}
where $e_{2,2}$ is the number of pairs of vertices of degree $2$ which are joined by an edge.
\end{thm}
	
\section{Conclusion, generalizations and open problems}

We end with some consideration and simulations leading to new open problems. All the simulations where done in Python 3.8.3. and plots in tikzplotlib v0.9.4.

\subsection{Comparing the bound from Theorem 3.1}

We calculated the values of both the quotient energy/size and the bound of Theorem 3.1 for $200$ random trees of size $n=2000$ following the Barabasi-Albert model with parameter $\alpha=1$. Results suggest that an increase or decrease in the energy results in a similar change in the bound from Theorem 3.1.  We believe that this means that for generic trees $\sum^n_{i=2} 2\sqrt{d_i-1}$ is proportional to the energy. 
Also, from the tables above we conjecture that there is a constant $c<1.5$ such that for any tree $c \mathcal{E}(T)\geq 2 \sum_i(d_i-1)$. 

\newpage
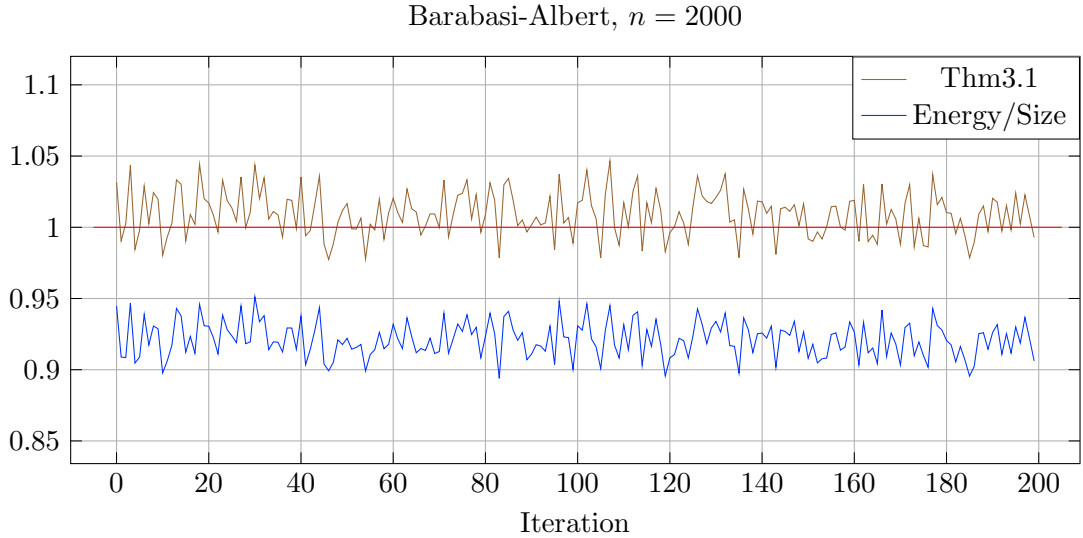
\begin{figure}[h]\label{Fig1}
\begin{tikzpicture}

\definecolor{color0}{rgb}{0,.2,1}
\definecolor{color1}{rgb}{.6,.4,.2}
\definecolor{color2}{rgb}{0.172549019607843,0.627450980392157,0.172549019607843}

\begin{axis}[
width=15cm,
height=7cm,
tick align=inside,
tick pos=both,
title={Barabasi-Albert, $n=2000$},
x grid style={white!69.0196078431373!black},
xlabel={Iteration},
xmajorgrids,
xtick style={color=black},
y grid style={white!69.0196078431373!black},
ymajorgrids,
ytick style={color=black},
legend style={at={(1,1)}},
xmin=-9.95, xmax=208.95,
ymin=0.834071832626361,ymax=1.12, 
]

\addplot [color1]
table {%
0 1.03147653780263
1 0.990547282415062
2 1.00261252002671
3 1.04340726495837
4 0.984537091002853
5 0.998053755516907
6 1.02889830401742
7 1.0025730371267
8 1.02435176659555
9 1.01968508959446
10 0.980379395310951
11 0.993311863076522
12 1.0027216993001
13 1.03330252581924
14 1.03022043513046
15 0.990481778149708
16 1.00889288055191
17 1.00214962549481
18 1.04426699692853
19 1.01996718767787
20 1.0170946617702
21 1.00850191291331
22 0.99662037291773
23 1.0330722447785
24 1.01869122640896
25 1.0136451763066
26 1.0041918007648
27 1.03526878817636
28 0.999701258302199
29 1.01064291634528
30 1.04366381429311
31 1.02009187963198
32 1.03468304265593
33 1.00587098569789
34 1.01103183469222
35 1.00873984960452
36 0.993137717156413
37 1.01964782036668
38 1.01892560314581
39 0.999057342048058
40 1.03525086055587
41 0.994042258855695
42 0.997636628300859
43 1.01629104673188
44 1.03595151969073
45 0.98813703028539
46 0.977419450988331
47 0.988267399308608
48 1.00341721074295
49 1.01186833926127
50 1.01661036374863
51 0.998959354237023
52 0.99875931044789
53 1.0063900324363
54 0.977681940768032
55 1.00211433970102
56 0.998149004852268
57 1.0197228777047
58 0.99164594851916
59 1.01003139719534
60 1.02037686494709
61 1.01020894233774
62 1.0034035837298
63 1.02683528398165
64 1.01304138739037
65 1.0106994743787
66 0.99465060239005
67 1.00088505776501
68 1.00942695485725
69 1.00931282660297
70 0.999831325082097
71 1.03311560294832
72 0.993110246450344
73 1.00885882352371
74 1.0225504314789
75 1.02381239901846
76 1.0332300616353
77 1.00580142077036
78 1.02286725961282
79 0.996483163055406
80 1.0087901424851
81 1.03117417578125
82 1.01910503174866
83 0.97860279052157
84 1.02977456555596
85 1.03432087525253
86 1.01906381545874
87 1.00114011327788
88 1.00511928236208
89 0.996750140562803
90 1.00162759471854
91 1.00711699533818
92 1.00168146243768
93 1.00309014365922
94 1.02169611573473
95 0.984106863779933
96 1.0372186598303
97 1.00303516991523
98 1.00682049210787
99 0.988972922096653
100 1.0171880645011
101 1.01895657079538
102 1.04010961309606
103 1.01509488048162
104 1.00609296669139
105 0.978400825621367
106 1.02295528245354
107 1.04648001251454
108 0.99962036459562
109 0.987442207434852
110 1.01721475784174
111 1.00032407424423
112 1.02526367186299
113 1.03603670946693
114 0.983285768338302
115 1.01716417439493
116 1.00326434900389
117 1.02720149491209
118 1.01210861822459
119 0.982812740738164
120 0.996249639671017
121 1.0004049774342
122 1.01089758778765
123 1.00289971654849
124 0.988024660680728
125 1.01457209993152
126 1.03604791950878
127 1.02199524076438
128 1.01817744686901
129 1.01670279723642
130 1.02094419018261
131 1.026014865999
132 1.03744914376666
133 1.00371075235416
134 1.00519936052039
135 0.978661139697279
136 1.02669099837926
137 1.01326500657219
138 0.995244183735137
139 1.01836098950522
140 1.01791325240224
141 1.00982779823281
142 1.01472408157708
143 0.980973154533577
144 1.01301317941078
145 1.01404615777793
146 1.01122629708875
147 1.01584541627922
148 1.00117988009994
149 1.01650761631004
150 0.991799778450635
151 0.990146670459595
152 0.996681602032744
153 0.991834041046493
154 1.00028321928053
155 1.01440286767547
156 1.01485946897476
157 1.00017093781889
158 0.998036361524703
159 1.01822670262006
160 1.01894551475204
161 0.990132131815675
162 1.03007496959
163 0.990021544416759
164 0.994391957238471
165 0.988051267816316
166 1.03038365791209
167 1.002216243355
168 1.01242917382634
169 1.00583893833306
170 0.987980230283248
171 1.01692063970931
172 1.02992595379423
173 0.986259025914432
174 1.00615999937486
175 0.987118529053392
176 0.986229744718951
177 1.0369717400159
178 1.01589404315551
179 1.02097194117825
180 1.01031786241946
181 1.00976464844579
182 0.995554641607478
183 1.00625430642979
184 0.993995223569609
185 0.978559352556126
186 0.989405216547566
187 1.00916531195219
188 1.01494175628878
189 0.99711318511211
190 1.02035709679069
191 1.01757769858646
192 0.99740182951294
193 1.01525856577941
194 0.998211519685206
195 1.02400267867027
196 1.00230759078848
197 1.02265209278062
198 1.00788962052478
199 0.992856106646352
};
\addlegendentry{Thm3.1}
\addplot [color0]
table {%
	0 0.944736052127596
	1 0.908872058061806
	2 0.908503753268027
	3 0.946745548671455
	4 0.904636091003337
	5 0.90888106884928
	6 0.938663125994264
	7 0.917423650508971
	8 0.930666913792919
	9 0.928764153118309
	10 0.897728524992735
	11 0.906250322893796
	12 0.917108216529508
	13 0.942989697177773
	14 0.937857449205368
	15 0.91228330033165
	16 0.923344680448353
	17 0.911564642355359
	18 0.945896305196653
	19 0.930888780963125
	20 0.930487561834502
	21 0.922550483723611
	22 0.911199538722856
	23 0.938373368107531
	24 0.928053746613269
	25 0.923804419869055
	26 0.918985933733919
	27 0.944688450999485
	28 0.918286880145527
	29 0.919538706449391
	30 0.951208412148625
	31 0.933665808555517
	32 0.937960890196981
	33 0.914123182406193
	34 0.919524459326946
	35 0.91933911221973
	36 0.912471794446802
	37 0.929348241581026
	38 0.929235969358352
	39 0.914350848675966
	40 0.937478313343917
	41 0.903504372432035
	42 0.914164428143177
	43 0.927674933805673
	44 0.943585621752385
	45 0.903946758369628
	46 0.899247641413809
	47 0.905107324134749
	48 0.92094369000863
	49 0.917684984572741
	50 0.922012433504727
	51 0.914393814464649
	52 0.915680335901973
	53 0.917732695410538
	54 0.899063396951165
	55 0.91058217074034
	56 0.914040996560829
	57 0.926268737099647
	58 0.914768480310512
	59 0.91779987792251
	60 0.931820194772348
	61 0.921649005398515
	62 0.914851256122977
	63 0.936344762632373
	64 0.923359208350725
	65 0.911906567904756
	66 0.914863102444869
	67 0.913533954370129
	68 0.922069158354546
	69 0.911390800034848
	70 0.912760212721818
	71 0.939542839154537
	72 0.911522428513004
	73 0.921591777662894
	74 0.932057297968888
	75 0.926923971765858
	76 0.938502982165274
	77 0.924688904451622
	78 0.929721289582172
	79 0.908486327307437
	80 0.923431090852545
	81 0.939602119736773
	82 0.925355796035348
	83 0.893954290825949
	84 0.937321515175208
	85 0.941021510555538
	86 0.927861896739494
	87 0.920678734153591
	88 0.92597869356175
	89 0.907040569669495
	90 0.910908923534854
	91 0.917543657859646
	92 0.916598774584579
	93 0.912986272016822
	94 0.930477897763589
	95 0.90368570673771
	96 0.94798468667587
	97 0.923018439947101
	98 0.922556129128524
	99 0.900343262436047
	100 0.930786634158167
	101 0.927799697063073
	102 0.946263687398413
	103 0.921626878620612
	104 0.916006409462132
	105 0.900855671072754
	106 0.927515222158863
	107 0.944929870442286
	108 0.917203369528416
	109 0.908285575640153
	110 0.931547211393036
	111 0.914013956328321
	112 0.938165560074833
	113 0.940642706886157
	114 0.90314369866535
	115 0.92816489979015
	116 0.916463909408247
	117 0.935314025361582
	118 0.917950120799184
	119 0.895575800103755
	120 0.908337988116502
	121 0.910634776612711
	122 0.922056544687248
	123 0.920358772982061
	124 0.908288287384544
	125 0.923610812617931
	126 0.942722124369889
	127 0.931880385542039
	128 0.918319129905876
	129 0.929299585764899
	130 0.933873397445966
	131 0.926528323817735
	132 0.939691312743216
	133 0.916870504962696
	134 0.916439291260051
	135 0.89769375220448
	136 0.936404049426637
	137 0.92822939265621
	138 0.912086600335993
	139 0.925288249906699
	140 0.925718067009992
	141 0.92219460910618
	142 0.93060654855905
	143 0.901706987202141
	144 0.927890521756088
	145 0.926807275049618
	146 0.924150594823705
	147 0.933912868092413
	148 0.912812807726606
	149 0.926551804859275
	150 0.907997806598916
	151 0.917799308956698
	152 0.904807194390539
	153 0.907693118757727
	154 0.908155497310839
	155 0.924773183902571
	156 0.925869028199144
	157 0.913742317439997
	158 0.915921913022492
	159 0.933522032034232
	160 0.926474840450034
	161 0.903408054047762
	162 0.933235623692794
	163 0.91191793756844
	164 0.915236798679612
	165 0.904338389239458
	166 0.941837281476997
	167 0.909206113274577
	168 0.925778996874116
	169 0.917980023461728
	170 0.903303873697385
	171 0.929657420935462
	172 0.932645238659266
	173 0.909803718795936
	174 0.919438367017815
	175 0.909915010296967
	176 0.901778398466262
	177 0.94288133688678
	178 0.930971986987404
	179 0.927870996179411
	180 0.920528227579453
	181 0.917459274403038
	182 0.905588628696257
	183 0.91621377083353
	184 0.906805979680404
	185 0.895500090301006
	186 0.902336135350172
	187 0.925262449247259
	188 0.926026044515238
	189 0.914415026615794
	190 0.926173146449536
	191 0.931659899714391
	192 0.910852517004978
	193 0.925304506602809
	194 0.911761613939036
	195 0.930034907944593
	196 0.918757319147974
	197 0.93661584716725
	198 0.92186558340518
	199 0.906209378499447
};
\addlegendentry{Energy/Size}
\addplot[color=red] 
coordinates {(-5,1) (205,1)};
\end{axis}

\end{tikzpicture}\vskip.3cm

\caption{ Energy/size for $200$ random trees of size $n=2000$ following the Barabasi-Albert compared with the bound from Theorem 3.1}
\end{figure}

\subsection{Varying $\alpha$}

When considering the general framework of preferential models, naturally, one is lead to ask how the energy changes when the paramater $\alpha$ varies.

Now, as the parameter $\alpha$ changes, the influence that the degree of a vertex has in the probability of its current degree to increase also changes. As $\alpha$ increases, so does this influence. Likewise, greater values of $\alpha$ intensify the distinction of the vertices of higher degrees when compared with the rest.  To be expilicit, as pointed out in \cite{KRL}, for $\alpha>1$ there are a few vertices which are linked to almost every other vertex and for $\alpha<1$ hubs are much smaller. In other words, when $\alpha$ is large, the preferential tree is similar to a star, while when $\alpha$ is small it is similar to a path.

Thus, recalling that among trees of constant size, the star is the one that has the lowest energy, and the path is the one with highest energy, the intuition leads to think that, as the value of $\alpha$ increases, the energy  decreases towards $\mathcal{E}(S_n)=2\sqrt{n-1}$. On the contrary, when $\alpha$ is small we expect the energy to increase towards $\mathcal{E}(P_n)\sim1.273n$. In Figure 4 we show a simulation where we calculate the acumulated avarege energy of $100$ realizations, of a BA random tree with $\alpha$ in $\{-5,-2,0,0.5,0.7,1,1.2,1.5,1.7,2\}$. This simulation strengthens this intuition. 

It may be possible to show by our same methods that a preferential model with parameter $\alpha>1$ is hypoenergetic. For $\alpha< 1$ the degree distribution degree distribution does not follow a power laws,  but a distribution of the form
$$q(d)=\frac{s}{d^\alpha}\prod^d_{i=1}\frac{i^\alpha}{s+i^\alpha}$$
where $s$ is such that $\sum^\infty_{d=1}q(d)=1$.

It would be interesting if there is a bound from below as a function of the degrees to show that for certain $\alpha$ the BA tree is not hypoenergetic.
\newpage

\begin{figure}[h]\label{Fig1}
\begin{tikzpicture}

\definecolor{color0}{rgb}{0.12156862745098,0.466666666666667,0.705882352941177}
\definecolor{color1}{rgb}{1,0.498039215686275,0.0549019607843137}
\definecolor{color2}{rgb}{0.172549019607843,0.627450980392157,0.172549019607843}
\definecolor{color3}{rgb}{0.83921568627451,0.152941176470588,0.156862745098039}
\definecolor{color4}{rgb}{0.580392156862745,0.403921568627451,0.741176470588235}
\definecolor{color5}{rgb}{0.549019607843137,0.337254901960784,0.294117647058824}
\definecolor{color6}{rgb}{0.890196078431372,0.466666666666667,0.76078431372549}
\definecolor{color7}{rgb}{0.737254901960784,0.741176470588235,0.133333333333333}
\definecolor{color8}{rgb}{0.0901960784313725,0.745098039215686,0.811764705882353}

\begin{axis}[
width=14cm,
height=9cm,
legend cell align={left},
legend style={fill opacity=0.8, draw opacity=1, text opacity=1, at={(0.09,0.5)}, anchor=west, draw=white!80!black},
tick align=inside,
tick pos=left,
title={Albert Barabasi, n=1000},
x grid style={white!69.0196078431373!black},
xlabel={Iteration},
xmajorgrids,
xmin=-4.95, xmax=103.95,
xtick style={color=black},
y grid style={white!69.0196078431373!black},
ylabel={Energy/size},
ymajorgrids,
ymin=0.0141298025129824, ymax=1.31038842992411,
ytick style={color=black},
legend style={at={(1,.689)}},
]
\addplot [color0]
table {%
0 1.25106538063863
1 1.2514675832236
2 1.25060444668747
3 1.25030176083044
4 1.25053430422407
5 1.2503644016052
6 1.25008473685082
7 1.24996881010317
8 1.25016981987839
9 1.25034256592929
10 1.25033379890363
11 1.25037514421234
12 1.2503123475325
13 1.25034025035676
14 1.25042066713076
15 1.25046003640694
16 1.25037617215585
17 1.25027740190653
18 1.25031955633615
19 1.25019919919678
20 1.25018027704643
21 1.25017236712285
22 1.25018752563013
23 1.2501869692505
24 1.25016872098736
25 1.25020776029734
26 1.25024626365414
27 1.2501599146221
28 1.24998314818952
29 1.25006482259629
30 1.24999955308545
31 1.25000112088292
32 1.250044338182
33 1.24999863271639
34 1.24986265513933
35 1.24985970350318
36 1.24983051711457
37 1.24987673540249
38 1.24987133843926
39 1.24992415838568
40 1.24985188701246
41 1.24987549817841
42 1.24987621968574
43 1.24984544402119
44 1.2498605741778
45 1.2498136413185
46 1.24975060959533
47 1.24976835638474
48 1.24975994505843
49 1.24983566114363
50 1.24985711097616
51 1.24984515104125
52 1.24985833057797
53 1.24989047489262
54 1.24991815275634
55 1.24988005902659
56 1.2498843149264
57 1.24991704694711
58 1.24988787450341
59 1.24992249994664
60 1.24991958722599
61 1.24991346006785
62 1.2498895673842
63 1.24991383695347
64 1.24991053205128
65 1.2498995157812
66 1.24995366197275
67 1.24996027614434
68 1.24995870550653
69 1.24989389194311
70 1.24986604257317
71 1.24986266185641
72 1.24987244949443
73 1.2498783636104
74 1.24988201875939
75 1.24988378552148
76 1.24984481024753
77 1.24985045667825
78 1.24986309665252
79 1.24985545240865
80 1.24985888538125
81 1.24986981445017
82 1.24986792702321
83 1.24986376784098
84 1.24983101062935
85 1.24983295867255
86 1.24983654296834
87 1.24985196321214
88 1.24986957735674
89 1.24985598809202
90 1.24984012720291
91 1.24986274751504
92 1.24985763532332
93 1.24985918873642
94 1.24981411342707
95 1.24982264629862
96 1.24981978414848
97 1.24981286804514
98 1.24985904946604
99 1.24984790519253
};
\addlegendentry{-5}
\addplot [color1]
table {%
0 1.21247014743684
1 1.21301541102344
2 1.21361942078301
3 1.21418535797234
4 1.21171035855936
5 1.21211960879269
6 1.21244497799023
7 1.21307129913116
8 1.21308450451762
9 1.21340822797476
10 1.21342936438853
11 1.21390551838772
12 1.21392772600266
13 1.21370470538126
14 1.21379439030726
15 1.21365899040438
16 1.21356346555341
17 1.21356531445938
18 1.21370607988431
19 1.21394470256327
20 1.21392852225539
21 1.21397429551209
22 1.21398192592727
23 1.21370853366327
24 1.2136657901899
25 1.2135574203224
26 1.21370049123458
27 1.21388325171282
28 1.21393529578385
29 1.21402249214542
30 1.21388491426716
31 1.21392694639656
32 1.21396761974556
33 1.21396939992611
34 1.2138265323023
35 1.21384469014635
36 1.21376694392843
37 1.21393795931514
38 1.21400029473389
39 1.2139438150128
40 1.21395344666956
41 1.21400320561392
42 1.21405682350734
43 1.21399149953898
44 1.21389793361066
45 1.213879330207
46 1.21390384212012
47 1.21382433933917
48 1.21391090743639
49 1.21403189244726
50 1.21395541764416
51 1.21393194000948
52 1.21403592440038
53 1.214071219594
54 1.21404687398082
55 1.21418216138287
56 1.214180525386
57 1.2141379784825
58 1.21407128270365
59 1.21404737208946
60 1.21406914430805
61 1.21411129780592
62 1.21414393360991
63 1.21410053005639
64 1.21396802601243
65 1.2139209914373
66 1.21398513287904
67 1.21394215301854
68 1.21384828204345
69 1.21382807357238
70 1.21384409796906
71 1.21386322906005
72 1.21383033274567
73 1.21378493957721
74 1.21384921064551
75 1.21387608763358
76 1.21382190407491
77 1.21379983547329
78 1.21378140249676
79 1.21375344360487
80 1.21373677645847
81 1.21378521621871
82 1.21375609659114
83 1.21377152401749
84 1.21385082636916
85 1.21387280168404
86 1.21394583385802
87 1.2139629631124
88 1.21394252447909
89 1.21397558476623
90 1.21395903279611
91 1.21395450816313
92 1.21394785817936
93 1.21394558838086
94 1.21388569409288
95 1.21393104929641
96 1.21392303254487
97 1.21392556170435
98 1.21393384914531
99 1.21391516314891
};
\addlegendentry{-2}
\addplot [color2]
table {%
0 1.12688340695944
1 1.13115370230364
2 1.13411443934254
3 1.13404202893921
4 1.13190348598687
5 1.13052379662014
6 1.13046100335723
7 1.13141272292958
8 1.13186730035392
9 1.13231161373986
10 1.13198422581032
11 1.1312201238252
12 1.1307267471658
13 1.13096758647548
14 1.13116503848986
15 1.13128907287804
16 1.13144369367809
17 1.13133242967758
18 1.13195932841521
19 1.13221417948847
20 1.13184727300408
21 1.13102452709641
22 1.13080849834362
23 1.13058009354331
24 1.13061445762141
25 1.13031782712676
26 1.1306598123692
27 1.13087602695687
28 1.13107261430381
29 1.13072415318057
30 1.13079713556781
31 1.13081563745205
32 1.13056601280967
33 1.13055909583133
34 1.13006448620004
35 1.12998578319762
36 1.13044201912189
37 1.13028907900779
38 1.12990214824194
39 1.12998904331356
40 1.12955945974326
41 1.12962894587904
42 1.12940915073327
43 1.12940537674513
44 1.1294069407782
45 1.12946979172728
46 1.12918330603447
47 1.12910001850958
48 1.12886624755516
49 1.12870086692263
50 1.12863238729812
51 1.12841611409183
52 1.12814346559951
53 1.12829914356166
54 1.12826438269065
55 1.12835165474147
56 1.12828879909109
57 1.12826058248672
58 1.12817378085581
59 1.12815609469706
60 1.12806449251065
61 1.1279621726803
62 1.12780158067644
63 1.12782335997336
64 1.12787265914776
65 1.12797312440984
66 1.12807821439803
67 1.12810460760129
68 1.12799662863632
69 1.12792600229894
70 1.12780756324741
71 1.12770139143146
72 1.12763869703861
73 1.1277118724678
74 1.12775227088632
75 1.12769400265948
76 1.12781227803853
77 1.12789539429768
78 1.12785926375994
79 1.12778592788627
80 1.12778343827194
81 1.12779973049147
82 1.12762772668993
83 1.12782330722813
84 1.12794571682297
85 1.12802461928842
86 1.1279496553995
87 1.1280675812562
88 1.12803603420904
89 1.12812242046784
90 1.12819916143195
91 1.12826477469373
92 1.12835081012877
93 1.12844171818568
94 1.12831023976818
95 1.12834816863889
96 1.12835162398712
97 1.12833476103095
98 1.12825874170298
99 1.12823989961324
};
\addlegendentry{0}
\addplot [color3]
table {%
0 1.07094303219878
1 1.06523649737307
2 1.06639486385219
3 1.06747520255996
4 1.06975067665877
5 1.07050357363029
6 1.07009938401155
7 1.06901757609494
8 1.06752799850984
9 1.06812651182656
10 1.06743035173731
11 1.06747063712028
12 1.06880146009242
13 1.06801280359506
14 1.06795192557159
15 1.06727800611644
16 1.06723145189937
17 1.06761695977506
18 1.0671884786939
19 1.06688038166982
20 1.06721451903156
21 1.06660155223981
22 1.06642537608871
23 1.06640820561433
24 1.06604186136108
25 1.06687327766016
26 1.0668536861128
27 1.06700224417835
28 1.06704542242344
29 1.06716410760853
30 1.06725585668997
31 1.06669118229855
32 1.06690427825368
33 1.06664287013489
34 1.0668640657514
35 1.06673663410873
36 1.06663559693221
37 1.06668440018188
38 1.0668148676378
39 1.06681001904173
40 1.06735622007974
41 1.06728291605757
42 1.06725581458344
43 1.06706394618649
44 1.06698557466639
45 1.06682302347417
46 1.06675089302142
47 1.06655230697298
48 1.06658723306367
49 1.06649251697437
50 1.06682413468904
51 1.06721385088454
52 1.06745767748643
53 1.06737845031017
54 1.06714338756942
55 1.06713146873491
56 1.06696769597653
57 1.06720808685885
58 1.06719961974682
59 1.06707092851811
60 1.06716980727837
61 1.06717472401244
62 1.06721578165141
63 1.06705960311786
64 1.06729438157018
65 1.06713878928616
66 1.06706537587919
67 1.0670798136626
68 1.06693766173816
69 1.06710423980104
70 1.06683212960488
71 1.0667883889688
72 1.06651680621632
73 1.06653086320791
74 1.06651841369259
75 1.06637403957893
76 1.06642043929068
77 1.06626625986413
78 1.06619199931714
79 1.06603560552819
80 1.06612017880229
81 1.06597806797033
82 1.06608054767473
83 1.06606858391001
84 1.06620159969483
85 1.06624591786998
86 1.06623390405362
87 1.06608039933774
88 1.06614042022586
89 1.06621103641756
90 1.06643367238715
91 1.06624126266151
92 1.06621560180739
93 1.06614304348459
94 1.06624151517416
95 1.06643265194034
96 1.06631315907236
97 1.06643197305307
98 1.06642395900225
99 1.06651944111657
};
\addlegendentry{0.5}
\addplot [color4]
table {%
0 1.03741819312817
1 1.03673013416122
2 1.03135578927556
3 1.02824697554282
4 1.02902133033419
5 1.03189361266628
6 1.02929567098244
7 1.02784820567772
8 1.02947492373392
9 1.02782893651115
10 1.02910519987223
11 1.02920291530872
12 1.02949102023271
13 1.02887357954878
14 1.02764160064057
15 1.02652540003141
16 1.02634723427735
17 1.02678848131318
18 1.02665796904182
19 1.02719855114011
20 1.02633244291813
21 1.02615072498165
22 1.02598284306332
23 1.02564948843316
24 1.02557881904412
25 1.02626705648237
26 1.02719731567357
27 1.02665039341908
28 1.02681667237769
29 1.02720671391219
30 1.02682433311757
31 1.0269688982965
32 1.02682330210107
33 1.02691765300421
34 1.02653335326181
35 1.02653566527745
36 1.02615522806098
37 1.02583918190007
38 1.02542828188539
39 1.02535204349666
40 1.02534309953373
41 1.02477724312246
42 1.02496129279448
43 1.02517006612216
44 1.02493579477328
45 1.02531313919095
46 1.02528505929667
47 1.02516571295018
48 1.02491815737386
49 1.02504085030447
50 1.02505422469632
51 1.02553344898223
52 1.02581033850726
53 1.02571598262307
54 1.02563680190161
55 1.02541070364681
56 1.02519394876244
57 1.02490083216194
58 1.02499326197022
59 1.02505075013601
60 1.02530418818977
61 1.0250574846061
62 1.02490441585838
63 1.02512437214142
64 1.0252858593999
65 1.02554175444515
66 1.02545471602537
67 1.02540928382055
68 1.02547200467278
69 1.02545301723294
70 1.0254184644926
71 1.02542135058145
72 1.0256460181536
73 1.0256843732341
74 1.0255789694395
75 1.02565267548571
76 1.02587083568676
77 1.02562758953985
78 1.02543760863044
79 1.02555344379258
80 1.02582396427941
81 1.0258955624469
82 1.02596348728573
83 1.02583627976469
84 1.02566174061724
85 1.02570359922614
86 1.02577335528102
87 1.02589591997212
88 1.02593046672932
89 1.02599471256875
90 1.0259667576601
91 1.0259373160511
92 1.02593762109058
93 1.02614682046879
94 1.02604547920939
95 1.02604813944312
96 1.02603021531553
97 1.02610328787316
98 1.02606563189676
99 1.02614034272633
};
\addlegendentry{0.7}
\addplot [color5]
table {%
0 0.948495621097878
1 0.938056134887736
2 0.931080841857424
3 0.925336296828634
4 0.931610154765828
5 0.932184290098879
6 0.932692294692321
7 0.931127175951757
8 0.934236920794576
9 0.934494577509055
10 0.932762065379638
11 0.930782951779073
12 0.931478925411793
13 0.931050703031848
14 0.930646052743402
15 0.929868161755228
16 0.928506673865289
17 0.928411252451424
18 0.928027779292665
19 0.928612836586328
20 0.928544829645015
21 0.927196565325986
22 0.927294777323362
23 0.92766563220884
24 0.927483700743049
25 0.927185234765477
26 0.926660120637513
27 0.926664657605508
28 0.92581627064747
29 0.925351435665641
30 0.924597800600752
31 0.924604501821061
32 0.924328452453987
33 0.923935471751591
34 0.924057607283388
35 0.924130578686619
36 0.923334621686074
37 0.923212925217236
38 0.923398012568838
39 0.923730032139901
40 0.924002203885892
41 0.924051579074415
42 0.92362190880306
43 0.923656380676105
44 0.923225654342271
45 0.923164533665693
46 0.923375922383425
47 0.923215720581659
48 0.923274916683964
49 0.923846827869357
50 0.92404413143369
51 0.923832425624646
52 0.923960681322617
53 0.923727024626043
54 0.923593298770712
55 0.923189868941893
56 0.923369391617102
57 0.923305454137683
58 0.923190876873078
59 0.923755002247687
60 0.924333763527447
61 0.924619120981437
62 0.924767781630077
63 0.924583969728203
64 0.924512953429654
65 0.924617038362552
66 0.924358296194723
67 0.924092004545562
68 0.923938832504545
69 0.923825124461494
70 0.923664988186168
71 0.923565672189132
72 0.923317308243214
73 0.92319612576062
74 0.923344913687033
75 0.923337936688854
76 0.923265033675007
77 0.923374447273617
78 0.923418208804716
79 0.923467704159036
80 0.923383866882732
81 0.923368563195107
82 0.923198397369061
83 0.923224198501151
84 0.923415167503423
85 0.92352482046989
86 0.923338583312511
87 0.923180705259357
88 0.922929634728727
89 0.922798988602428
90 0.922951535818901
91 0.923251496353774
92 0.923632514648418
93 0.92361502647313
94 0.923311868834186
95 0.923147998299732
96 0.923143497868504
97 0.923211353734791
98 0.923358184060856
99 0.923521839318069
};
\addlegendentry{1}
\addplot [color6]
table {%
0 0.628635216969408
1 0.708751870786235
2 0.734306409621358
3 0.758238721055647
4 0.750476949728678
5 0.758472826289911
6 0.757965420131123
7 0.745277940374444
8 0.751757690398112
9 0.746281057045292
10 0.740526806663623
11 0.747379582789449
12 0.750299464099719
13 0.751821096581405
14 0.750993000833407
15 0.749667476281127
16 0.749140884718573
17 0.750595607640001
18 0.749986663297202
19 0.751992060053915
20 0.751988091010654
21 0.751342926720309
22 0.751703017956584
23 0.754170316078853
24 0.752549031416715
25 0.753798500327051
26 0.753080468502758
27 0.753488437571683
28 0.754103444270978
29 0.754723632301048
30 0.754696455691991
31 0.754573185813957
32 0.754866967490128
33 0.755941388405566
34 0.756021532216562
35 0.757462984783042
36 0.755983384616291
37 0.754343239848923
38 0.754102395561741
39 0.753324763266239
40 0.755304171122424
41 0.756918405471254
42 0.756338552626363
43 0.757501477662367
44 0.756881829133712
45 0.756382713496244
46 0.755937431478886
47 0.755874751746734
48 0.755363585701492
49 0.753891177696075
50 0.753217574277513
51 0.753195705078043
52 0.752192897663246
53 0.75190092275508
54 0.751451001048658
55 0.750627615304101
56 0.749983751016138
57 0.750268251721179
58 0.749526768443717
59 0.748235012765851
60 0.745934902622842
61 0.745675070980751
62 0.744962875091858
63 0.744215200853492
64 0.744701758891858
65 0.743063252563088
66 0.743507312284589
67 0.744017573970209
68 0.744501341426856
69 0.745271028507337
70 0.744080673380607
71 0.744292689150204
72 0.74514848717364
73 0.745100227451453
74 0.744516834725978
75 0.744508562002797
76 0.744409860523953
77 0.744487805689381
78 0.745456834961415
79 0.745163661382376
80 0.744612926009735
81 0.745167916624471
82 0.745419916955299
83 0.746406525928038
84 0.746701517149009
85 0.746881138892577
86 0.746851968309514
87 0.746387398117732
88 0.745656546006271
89 0.745066716359571
90 0.745287522572488
91 0.744333554864152
92 0.743838613823853
93 0.743696326389938
94 0.744522119415364
95 0.745015698769983
96 0.745297550950918
97 0.746060227853351
98 0.745839233263988
99 0.745473188878261
};
\addlegendentry{1.2}
\addplot [white!49.8039215686275!black]
table {%
0 0.283978321136175
1 0.256675626323097
2 0.251430635879974
3 0.246342815528094
4 0.244807725410435
5 0.255339258596202
6 0.253099705676376
7 0.253884038213146
8 0.251607274427014
9 0.242211659591354
10 0.240364814122513
11 0.242527101933511
12 0.249426088275502
13 0.25300884381489
14 0.256091805725464
15 0.262917672622006
16 0.26691297421361
17 0.262347647486798
18 0.265444356194296
19 0.265975950200009
20 0.263092485561089
21 0.260458077880924
22 0.262116015749539
23 0.26205030603855
24 0.260493810433502
25 0.259636357760226
26 0.258140764445675
27 0.255413159985554
28 0.252491728449902
29 0.252818597420885
30 0.252264274625079
31 0.251347618481147
32 0.251127928098791
33 0.25208086177974
34 0.251536593314547
35 0.2526893173039
36 0.253662861736509
37 0.25207841692137
38 0.251939441599119
39 0.250850002598216
40 0.249059258822451
41 0.247497720300653
42 0.249776228109739
43 0.248659681091257
44 0.247783608912019
45 0.249930851552928
46 0.250635242993235
47 0.249203810247555
48 0.24820638662723
49 0.247716197763436
50 0.245750336602033
51 0.244714882065139
52 0.243818850549596
53 0.243588737849901
54 0.244285558961105
55 0.244878900479216
56 0.245816194371735
57 0.24599936908463
58 0.246346451873609
59 0.246824988617982
60 0.246036478893542
61 0.245228456968446
62 0.24531857087111
63 0.245588934776176
64 0.2445879513021
65 0.245561431147784
66 0.24484660280765
67 0.244074290611505
68 0.244698254075187
69 0.244104821116097
70 0.243149561262112
71 0.242555729835486
72 0.242582066079918
73 0.241874345095552
74 0.243264407458828
75 0.243212114347733
76 0.243073851114378
77 0.2429405707058
78 0.242674462071051
79 0.243926665020089
80 0.243411109051352
81 0.24380608683631
82 0.243823279761179
83 0.243288547536249
84 0.242939940650402
85 0.242682546543667
86 0.24242873061335
87 0.242990903471399
88 0.242396592889475
89 0.243276789045343
90 0.242513238045764
91 0.243810208224117
92 0.24343875708258
93 0.242766181497067
94 0.242791419595871
95 0.244262034970691
96 0.243817284825882
97 0.243233738404181
98 0.242601307809792
99 0.242338152673257
};
\addlegendentry{1.5}
\addplot [color7]
table {%
0 0.150749173852652
1 0.141999254225488
2 0.129736075911558
3 0.127185576690885
4 0.123285167664793
5 0.118478586942696
6 0.116428655459643
7 0.114406561973524
8 0.113817749054806
9 0.112597527768387
10 0.113166769229422
11 0.114359948256574
12 0.112639387228649
13 0.116212580829446
14 0.116439909708947
15 0.117175185727564
16 0.117525281128817
17 0.116300873112101
18 0.118254434600528
19 0.118891560270326
20 0.119912866806444
21 0.120685444474981
22 0.121017389723772
23 0.119822855634348
24 0.120143328223885
25 0.120189283498128
26 0.119428541070903
27 0.119205175802322
28 0.118603779216424
29 0.118311915901404
30 0.118394299594003
31 0.117537735627329
32 0.118724131261167
33 0.119423944371909
34 0.119388623075652
35 0.119179514604507
36 0.119085124886737
37 0.119053212244087
38 0.11955103797221
39 0.119853612380059
40 0.119940791501351
41 0.120040730345139
42 0.120184034082461
43 0.122055443127128
44 0.121744983299124
45 0.121216896802679
46 0.121139630911132
47 0.120728960571127
48 0.121209168523055
49 0.121075889231375
50 0.120718076343994
51 0.120834297811928
52 0.12115483934672
53 0.121215360281573
54 0.121247595955036
55 0.120988202546787
56 0.120557998804704
57 0.120156201916694
58 0.120162543834347
59 0.119722292123797
60 0.119767340313978
61 0.119744099984257
62 0.119946960061793
63 0.119620369255522
64 0.119885529182937
65 0.12039139053517
66 0.120636632411198
67 0.120215911106917
68 0.119977925538895
69 0.120323737852378
70 0.120552621361498
71 0.120426385232285
72 0.120403329247023
73 0.120175444185007
74 0.119972395999459
75 0.119792576484405
76 0.119724612061868
77 0.119683400273998
78 0.119614947194031
79 0.119485976801819
80 0.119359581527615
81 0.119092428148774
82 0.119102982550462
83 0.119044777001526
84 0.119102044780675
85 0.118907969230188
86 0.118822123878993
87 0.118690360277228
88 0.118834554533955
89 0.118742062223591
90 0.1191590455667
91 0.119377637850349
92 0.11920065409019
93 0.119673482923274
94 0.119419571812334
95 0.11930702480909
96 0.119185420590903
97 0.119147239644692
98 0.119033463320763
99 0.119152014210726
};
\addlegendentry{1.7}
\addplot [color8]
table {%
0 0.0730506492134881
1 0.0731990923614922
2 0.0740708051902419
3 0.0812335040816388
4 0.0816221483582992
5 0.0808387446654309
6 0.0814451473274999
7 0.0799040056047957
8 0.0802053100266779
9 0.0838138843402676
10 0.0837507248297912
11 0.0830841170784506
12 0.0828275481475662
13 0.0823835225746099
14 0.0823997614454609
15 0.0822336371763727
16 0.0812305587828055
17 0.0816677652115698
18 0.082235168582687
19 0.0823474460753471
20 0.0820688750800536
21 0.0816951515785422
22 0.0815596418552118
23 0.0812397714681893
24 0.0809908981871033
25 0.0806243952801617
26 0.0802710228835086
27 0.0801821352835638
28 0.0798683837417349
29 0.0795755489693613
30 0.0796310835642204
31 0.0795366904981161
32 0.0793997586479016
33 0.0791551581162141
34 0.079659201847457
35 0.0795534038155397
36 0.0794781845370438
37 0.080185905535691
38 0.0803570032984329
39 0.0802926153452005
40 0.0804186515608445
41 0.080603317020021
42 0.0804276735826598
43 0.0801258780051699
44 0.0802560980734398
45 0.0803183387140103
46 0.0801637070225099
47 0.0800738479740916
48 0.0803420191533805
49 0.0802280574463825
50 0.0802929887584467
51 0.0802941508063966
52 0.0808352684196799
53 0.0807058552646266
54 0.0807465860695619
55 0.0806233784262983
56 0.08067856265735
57 0.0807608214477201
58 0.0808561509206879
59 0.0808557830181657
60 0.080676382593964
61 0.0805922379351346
62 0.0805164074089634
63 0.0804246393260519
64 0.0803839754159455
65 0.0803426290301221
66 0.0803042944332566
67 0.0803167259749822
68 0.0803291046846345
69 0.0801408127344698
70 0.0800132427039917
71 0.0801041450140025
72 0.0800402461145549
73 0.0799488954459907
74 0.0800881280800105
75 0.0801702092075078
76 0.0803104723870472
77 0.0801669536839551
78 0.0800270683657515
79 0.0800382262523344
80 0.0799519598691388
81 0.0799112104249893
82 0.0799232848019342
83 0.079818047370608
84 0.079782497840072
85 0.0798434196578235
86 0.0798675650528458
87 0.0800711030960893
88 0.0799480122022666
89 0.0799306463644793
90 0.0798637926467026
91 0.0799926487047322
92 0.0802281626060778
93 0.0801602774511919
94 0.0802556932856098
95 0.0801191617740212
96 0.0800260089781089
97 0.0799548318376536
98 0.0800415952724925
99 0.0799520127254728
};
\addlegendentry{2}
\end{axis}

\end{tikzpicture}
\vskip.3cm
\caption{ Energy/size for $100$ random graphs of size $n=1000$ following the Barabasi-Albert model with $10$ different parameters $\alpha$. Results suggest that energy is decreasing with $\alpha$.}
\end{figure}
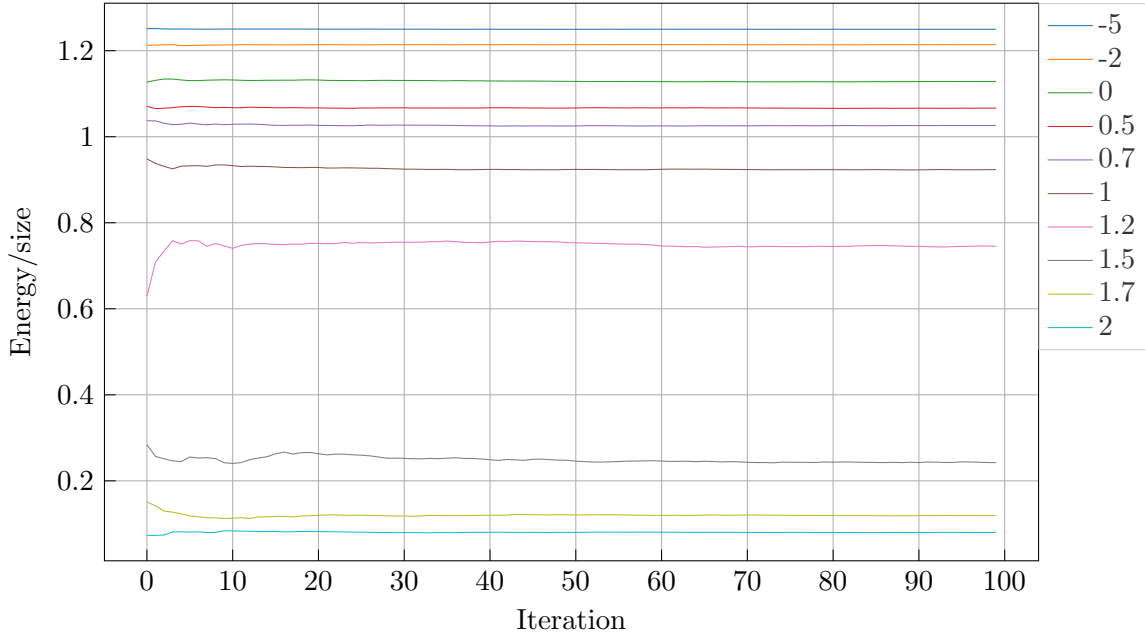

Moreover, in Figure 5, we preset the plot of the average ratio between energy and the size of $50$ BA random trees for various values of $\alpha$ ratio.  The energy appear to be decreasing with respect to $\alpha$. The value of $\alpha$ for which graphs start to be hypoenergetic in our simulations was $\tilde\alpha=0.8006376$. 

Thus let us state precisely the above observations as a conjecture.

\begin{conj}
Let $\{X_n\}_n>0$ be a sequence of trees with parameter $\alpha$.
\begin{enumerate}
\item  For every $\alpha$ there exists $g(\alpha)$ such that following limit exists almost surely
$$\lim_{n\to \infty}\frac{\mathcal{E}(X_n)}{n}=g(\alpha). $$ 
\item  For $\alpha<.79$ $g(\alpha)>1$. In particular, taking $\alpha=0$, this would imply that random recursive random trees are not hypoenergetic.
\item Para $\alpha>.81$, $g(\alpha)<1$.
\item  $g(\alpha)$ is strictly decreasing on $\alpha$ and continuous.
\item  $g(\alpha)\to0$ as $\alpha\to\infty$  and $g(\alpha)\to1.273$ as $\alpha\to\infty$.
\end{enumerate}\end{conj}

We note that since the degree distribution does not determine the energy, proving the monotonicity on $\alpha$ needs new ideas or tools.

\newpage

\begin{figure}[h] \label{Monotonicity}

\begin{tikzpicture}

\definecolor{color0}{rgb}{0.12156862745098,0.466666666666667,1}

\begin{axis}[
height=9cm,
width=14cm,
tick align=inside,
tick pos=left,
title={Barabasi Albert, n=1000},
x grid style={white!69.0196078431373!black},
xlabel={$\alpha$},
xmajorgrids,
xmin=-2.1, xmax=5.1,
xtick style={color=black},
y grid style={white!69.0196078431373!black},
ylabel={Energy/Size},
ymajorgrids,
ymin=0.00579808008528975, ymax=1.27067790971786,
ytick style={color=black},
xtick={-2,-1,0,1,2,3,4,5},
extra x ticks={0.800637623560177616191853298},
extra x tick labels={$\tilde\alpha$}
]
\addplot [color0,only marks]
table {%
-2 1.21318337200729
-1.91139240506329 1.21089627452603
-1.82278481012658 1.20966601601875
-1.73417721518987 1.20686360001265
-1.64556962025316 1.20525388447232
-1.55696202531646 1.20245685203164
-1.46835443037975 1.20021337812354
-1.37974683544304 1.19776221957619
-1.29113924050633 1.19477013393243
-1.20253164556962 1.19213375423888
-1.11392405063291 1.18954037598238
-1.0253164556962 1.18460243931716
-0.936708860759493 1.18251826561115
-0.848101265822784 1.17890299846898
-0.759493670886076 1.17427547861974
-0.670886075949367 1.16999762232704
-0.582278481012658 1.16633577018957
-0.493670886075949 1.16231229004385
-0.40506329113924 1.15557642429919
-0.316455696202531 1.15184066489431
-0.227848101265822 1.14535245990629
-0.139240506329113 1.13937175911711
-0.0506329113924045 1.1294645733579
0.0379746835443043 1.12494241801603
0.126582278481013 1.11389421202261
0.215189873417722 1.10713203576014
0.303797468354431 1.09554908175684
0.39240506329114 1.08275925533306
0.481012658227849 1.06938412827562
0.569620253164558 1.05389142054227
0.658227848101266 1.03618939159151
0.746835443037975 1.01523701849108
0.835443037974684 0.990142950752623
0.924050632911393 0.95802062990086
1.0126582278481 0.921341039251881
1.10126582278481 0.859580063011068
1.18987341772152 0.7555278854036
1.27848101265823 0.585269593509467
1.36708860759494 0.432086165731238
1.45569620253165 0.299785890983952
1.54430379746836 0.185803332058609
1.63291139240506 0.151158197942669
1.72151898734177 0.115511549050222
1.81012658227848 0.0991877864508556
1.89873417721519 0.0880144329310653
1.9873417721519 0.0829109199118025
2.07594936708861 0.076687333074255
2.16455696202532 0.0729281590109873
2.25316455696203 0.0688837498568009
2.34177215189874 0.0690265327864508
2.43037974683544 0.0673213977720435
2.51898734177215 0.0665418458943931
2.60759493670886 0.0667084350928719
2.69620253164557 0.0659929485734555
2.78481012658228 0.0659445504605239
2.87341772151899 0.0647124747414721
2.9620253164557 0.0650634517991373
3.05063291139241 0.0647281540869593
3.13924050632912 0.0640766568344299
3.22784810126582 0.0639544844307279
3.31645569620253 0.0642650742088458
3.40506329113924 0.063820063941223
3.49367088607595 0.0640427626235379
3.58227848101266 0.0636672549989383
3.67088607594937 0.063817530136514
3.75949367088608 0.0637179625354343
3.84810126582279 0.0635605583960136
3.9367088607595 0.0634265882787498
4.0253164556962 0.0637560302389577
4.11392405063291 0.0637338768838993
4.20253164556962 0.0636073985452147
4.29113924050633 0.0633872406393775
4.37974683544304 0.0632926177958612
4.46835443037975 0.0633872406393775
4.55696202531645 0.0633993147449125
4.64556962025316 0.063560558049887
4.73417721518987 0.0634386623842848
4.82278481012658 0.0632926177958612
4.91139240506329 0.0634386623842848
5 0.0634265882787498
};
\addplot[color=red] 
coordinates {(0.800637623560177616191853298,0) (0.800637623560177616191853298,1.2706)};
\end{axis}

\end{tikzpicture}
\vskip.3cm

\caption{For $80$ values of $\alpha$ in the interval $[-2,5]$, we calculated the average ratio of Energy/size of $50$ random graphs following the BA-model with parameter $\alpha$ of size $n=1000$. 
}
\end{figure}
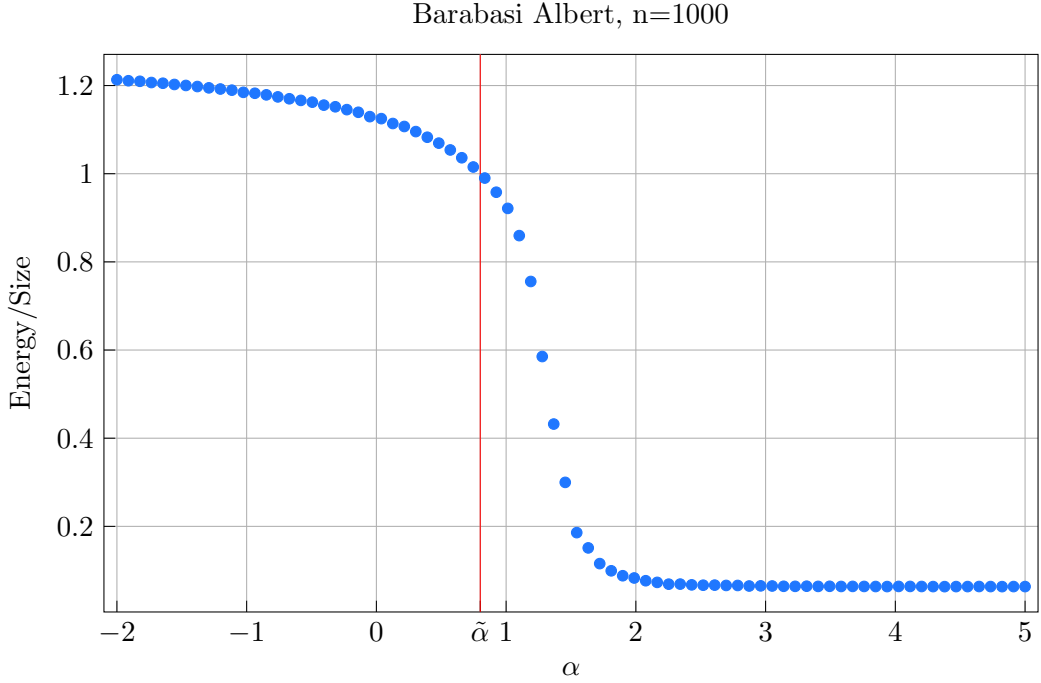

\subsection{Generic trees and Erd\"os-R\'enyi}

From the main result of this paper one may be lead to the conclusion that a typical tree is hypoenergetic. However, this is not true, since the preferential model of Barabasi-Albert does not choose a tree with uniform distribution, but favors trees with certain large hubs. 
\begin{figure}[h] 
\begin{tikzpicture}

\definecolor{color0}{rgb}{0.12156862745098,0.466666666666667,0.705882352941177}
\definecolor{color1}{rgb}{.8,0.4,0.35}
\begin{axis}[
width=15cm,
height=5cm,
tick align=inside,
tick pos=both,
title={Erdos Renyi $p=0.001$ y BA $\alpha=0$, $n=3000$},
x grid style={white!69.0196078431373!black},
xlabel={Iteration},
xmajorgrids,
xtick style={color=black},
y grid style={white!69.0196078431373!black},
ymajorgrids,
ytick style={color=black},
ymin=0.98, ymax=1.2,
xmin=-3, xmax=203,
legend style={at={(1,.315)}},
]
\addplot [color1]
table {%
	0 1.13728655134136
	1 1.12436046300123
	2 1.13204104729774
	3 1.12578994171712
	4 1.13028580889531
	5 1.13461953736855
	6 1.12840042232226
	7 1.12574623922613
	8 1.12750797272678
	9 1.12174673596828
	10 1.13044892845514
	11 1.12923482440545
	12 1.12651532691781
	13 1.12301641913804
	14 1.1254438415615
	15 1.13394343315649
	16 1.13049657579879
	17 1.13087295700704
	18 1.12346652861161
	19 1.12941251996334
	20 1.12989970187006
	21 1.12675839653298
	22 1.12683642552726
	23 1.13024549585901
	24 1.13100259687667
	25 1.12086833644673
	26 1.13025875341396
	27 1.12570132196627
	28 1.12443898749887
	29 1.12647536799515
	30 1.12684048099055
	31 1.12429257358391
	32 1.12246030130469
	33 1.12811926316089
	34 1.12381429123385
	35 1.12816828148865
	36 1.12517768654753
	37 1.1314761966239
	38 1.1273790717587
	39 1.12460104946806
	40 1.13102018850105
	41 1.12573575885913
	42 1.12375817014996
	43 1.12664191826047
	44 1.12683468910825
	45 1.12365277275531
	46 1.12158809988304
	47 1.12919120110408
	48 1.12888617978924
	49 1.12986622870456
	50 1.12868578252832
	51 1.13547858555613
	52 1.13200337327445
	53 1.1233222902577
	54 1.12884589237182
	55 1.13109425265355
	56 1.12636486948867
	57 1.12759520994699
	58 1.11655135014154
	59 1.12423113394227
	60 1.1229407113903
	61 1.12691021364222
	62 1.13146250644389
	63 1.13110190908521
	64 1.13079771254292
	65 1.12105281226887
	66 1.12943961025335
	67 1.12881984091035
	68 1.12854537691721
	69 1.12307328341637
	70 1.12834521835569
	71 1.12668063851189
	72 1.12873860185289
	73 1.12907846083191
	74 1.12714103563727
	75 1.13203776799292
	76 1.12025802004212
	77 1.12918682600507
	78 1.12424791358649
	79 1.13082697688517
	80 1.12823536358168
	81 1.12570129514307
	82 1.12491570622248
	83 1.13259382332598
	84 1.12261848122063
	85 1.12905357277678
	86 1.12801458952282
	87 1.12647079777086
	88 1.12596847791055
	89 1.12801726193499
	90 1.12796596254944
	91 1.12638551953248
	92 1.13348231455199
	93 1.12475403674911
	94 1.124907931365
	95 1.12447432705557
	96 1.13480208871474
	97 1.13011078428009
	98 1.13365924340192
	99 1.11928185996483
	100 1.13378672994895
	101 1.13335001251491
	102 1.12732900294306
	103 1.12767064480319
	104 1.13311547740744
	105 1.12277763306303
	106 1.12764094029827
	107 1.12635725916458
	108 1.1323962878215
	109 1.12066017991179
	110 1.12625021236252
	111 1.11932801421926
	112 1.13153516973884
	113 1.12323233114396
	114 1.12693718124403
	115 1.13138002842795
	116 1.12758458169136
	117 1.12649515406147
	118 1.12808875045739
	119 1.12712141427805
	120 1.12660626795406
	121 1.13060331874509
	122 1.12321318542349
	123 1.12252040933752
	124 1.12927894613581
	125 1.12689999862619
	126 1.12084385478733
	127 1.121316035049
	128 1.12796646745988
	129 1.12904566478314
	130 1.1242416306311
	131 1.12853891191162
	132 1.13186529344247
	133 1.12539213749627
	134 1.12795142019501
	135 1.12777380881384
	136 1.13944221811748
	137 1.12813427754922
	138 1.12472404903324
	139 1.13257591000403
	140 1.12689697125308
	141 1.12846558747285
	142 1.13186072352428
	143 1.13520333824561
	144 1.12832289362885
	145 1.12829220998198
	146 1.12722395410708
	147 1.11774388423672
	148 1.12544773773995
	149 1.12442525069001
	150 1.12804077545036
	151 1.12437483031672
	152 1.13143854640162
	153 1.12428232183632
	154 1.13519588718226
	155 1.12764634112353
	156 1.1274894755976
	157 1.12394400755864
	158 1.13641076442854
	159 1.12322454113132
	160 1.12332009913152
	161 1.12805173197966
	162 1.13215180081534
	163 1.12704956501259
	164 1.12431785921693
	165 1.13153251873978
	166 1.12885723362351
	167 1.12378643706963
	168 1.13189911957834
	169 1.12643236573935
	170 1.13085123508652
	171 1.12475753936685
	172 1.12483727288167
	173 1.12917798665892
	174 1.1231659432064
	175 1.13459178844621
	176 1.12144850384547
	177 1.12698322429507
	178 1.13025125647668
	179 1.12663724016932
	180 1.1292018776735
	181 1.12393077613618
	182 1.12843939810375
	183 1.12918065897838
	184 1.12728664080068
	185 1.12784495291883
	186 1.12643150504089
	187 1.12221278762177
	188 1.1340469075824
	189 1.1342717142822
	190 1.133011151986
	191 1.13110101835411
	192 1.12707855820176
	193 1.13239778896807
	194 1.12630414018988
	195 1.12463643880795
	196 1.12784845139902
	197 1.12420346376586
	198 1.12456750528849
	199 1.12765557186433
};
\addlegendentry{BA $\alpha=0$}
\addplot [color0]
table {%
	0 1.10488987517787
	1 1.08669595357476
	2 1.08234114901359
	3 1.10551418665283
	4 1.10021940649222
	5 1.09047114067758
	6 1.07673387721773
	7 1.09615370351271
	8 1.1094864435525
	9 1.10749124134275
	10 1.09241110045068
	11 1.08619983393587
	12 1.09831512445894
	13 1.08660187291786
	14 1.0727717888609
	15 1.11103376736513
	16 1.12416576240695
	17 1.11193273241572
	18 1.09136394396587
	19 1.07719074846675
	20 1.09372583217784
	21 1.08697932193879
	22 1.07904426838423
	23 1.11543318910217
	24 1.10264685067636
	25 1.08845321664415
	26 1.09578617499477
	27 1.09058518587652
	28 1.08923326724445
	29 1.07553227308647
	30 1.09201813669337
	31 1.09342409487652
	32 1.08293800482139
	33 1.11742650432763
	34 1.08670504264326
	35 1.08774165728754
	36 1.08844388055358
	37 1.08360322232719
	38 1.06848784223414
	39 1.10352512080923
	40 1.09115224621294
	41 1.09603740720699
	42 1.09852416736929
	43 1.08858341713076
	44 1.07260564411031
	45 1.08959383614007
	46 1.08688807264351
	47 1.10145431399416
	48 1.09379823326257
	49 1.09866366627051
	50 1.08205123188871
	51 1.08561906486738
	52 1.10038387509471
	53 1.1067486054168
	54 1.09322819935113
	55 1.10616794598055
	56 1.11694403703376
	57 1.07315214601976
	58 1.09941618780691
	59 1.09787254969514
	60 1.08174215955534
	61 1.09926702106978
	62 1.11863062984489
	63 1.10106474137805
	64 1.0861651628969
	65 1.07600483133727
	66 1.07614537682591
	67 1.10852841653338
	68 1.08622110260165
	69 1.10430805719775
	70 1.08564746601444
	71 1.07056450842288
	72 1.10471611487952
	73 1.08666315779226
	74 1.08683047659511
	75 1.09372436789755
	76 1.09238269471974
	77 1.08660848290634
	78 1.11282649266542
	79 1.0973968057725
	80 1.07763917936073
	81 1.09962992380516
	82 1.07590668998431
	83 1.09043539573625
	84 1.11615680198941
	85 1.08400131193354
	86 1.09345096231687
	87 1.10047506105004
	88 1.09069950134359
	89 1.09150284173457
	90 1.10315041792046
	91 1.10125434737773
	92 1.09786652327807
	93 1.06696236777567
	94 1.08234229377682
	95 1.08899651101818
	96 1.11081806347026
	97 1.09255501075386
	98 1.07232106379018
	99 1.08961657624538
	100 1.09036763303135
	101 1.0946542096874
	102 1.10280485715719
	103 1.10091720888831
	104 1.09235976621621
	105 1.06330125632488
	106 1.09215908057138
	107 1.07587707506139
	108 1.11907697254024
	109 1.11394365998018
	110 1.0996667124378
	111 1.10006312263945
	112 1.09921233762992
	113 1.0910297999257
	114 1.08734119631719
	115 1.09345231609271
	116 1.10424794374334
	117 1.09750121241397
	118 1.09370967333723
	119 1.09339876772124
	120 1.08246317405753
	121 1.07008446627524
	122 1.08188062138484
	123 1.09995872252535
	124 1.07919020926904
	125 1.08540870059684
	126 1.09862188500123
	127 1.09563313001496
	128 1.08784369528418
	129 1.0878527125239
	130 1.09645950431284
	131 1.09927953421736
	132 1.08448014507694
	133 1.0977316071534
	134 1.09446418258982
	135 1.11125262924467
	136 1.08998253099052
	137 1.09721229726377
	138 1.089049422024
	139 1.07244926997837
	140 1.11759383547656
	141 1.10342229179602
	142 1.08749579277913
	143 1.08646754617153
	144 1.08450335957023
	145 1.07969057438638
	146 1.07513809734835
	147 1.08903207832198
	148 1.11210358880347
	149 1.09732564233482
	150 1.0838103981523
	151 1.10918580651629
	152 1.1091574067044
	153 1.10401654674431
	154 1.09888376919726
	155 1.07485998352554
	156 1.08567292196124
	157 1.08738354691087
	158 1.07056615072162
	159 1.09408526893388
	160 1.106060601673
	161 1.10309125199025
	162 1.11023530215566
	163 1.10735351181974
	164 1.10362981342184
	165 1.12390858879985
	166 1.08954662044682
	167 1.10321003303638
	168 1.10549467054623
	169 1.11275471140791
	170 1.08957667168499
	171 1.10378662298644
	172 1.07992082888277
	173 1.09470238851788
	174 1.09291635637914
	175 1.07086431742603
	176 1.10388490650881
	177 1.07852834678699
	178 1.07669025741558
	179 1.06882843314437
	180 1.09774891731468
	181 1.10453673683652
	182 1.08715899154729
	183 1.07733213705273
	184 1.11026420422548
	185 1.10139384946199
	186 1.09889067973453
	187 1.12473610478598
	188 1.09590688281215
	189 1.09243578928987
	190 1.09959293159971
	191 1.0974139430217
	192 1.08590952768341
	193 1.10489070015226
	194 1.12164454091487
	195 1.08565698939159
	196 1.0996632537806
	197 1.1060230809149
	198 1.11524967452648
	199 1.11071911940218
};
\addlegendentry{Erd\"os-R\'enyi}
\addplot[color=red] 
coordinates {(-5,1) (205,1)};
\end{axis}

\end{tikzpicture}

\caption{Energy/size of $200$ random graphs of size $n=3000$ following the Erd\"os-R\'enyi model with parameter $p=2/n=0.000\bar6$ so that the expected number of edges is that of a tree of size $3000$. }
\end{figure}
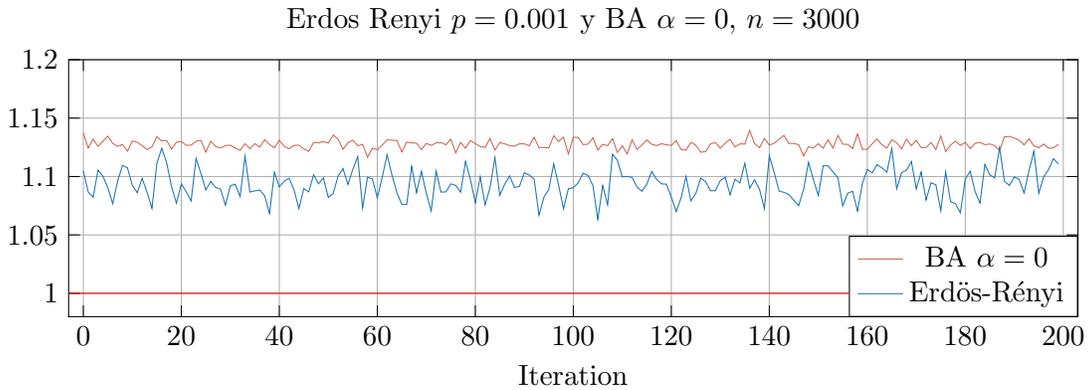

The correct model to choose a tree with uniform distribution is the random recursive tree, which corresponds to $\alpha=0$ and as can be seen from Figure 4 seems to be  not hypoenergetic. It is natural to compare this model with an Erd\"os-R\'enyi model with expected number of edges equal $n-1$. This is done in Figure 6.  In all of the instances, the resulting graph was also not hypoenergetic, but with smaller energy than that of uniform trees.  We conjecture that this is the case in general.  Of couse, similarly as for $\alpha$, one may vary $p$ to analyze how the energy varies.

\subsection{Conclusions}

In this paper we have shown that the energy of a BA tree is hypoenergetic. This result leads to new questions and conjectures on the behavior of other models for random trees or random graphs, which differ from the usual Erd\"os-R\'enyi graphs.

\end{document}